\documentclass{article}
\usepackage[margin=0.8in]{geometry}
\usepackage{hyperref} 
\usepackage{breakurl} 
\usepackage{algorithm} 
\usepackage{algorithmic} 
\usepackage{caption} 
\usepackage{subcaption} 
\usepackage{multirow}
\usepackage{soul,color}
\usepackage{amsthm}
\usepackage{amsopn}
\usepackage{amsmath}
\usepackage{amsfonts}
\usepackage{graphicx}

\usepackage{lmodern}
\usepackage[T1]{fontenc}

\usepackage{soul,xcolor}

\theoremstyle{plain} 
\newtheorem{thm}{Theorem}[section] 
\newtheorem{lem}[thm]{Lemma} 
\newtheorem{cor}[thm]{Corollary}

\theoremstyle{remark}

\DeclareMathOperator{\sgn}{sgn} 

\newcommand{\fraction}[2]{\textstyle\frac{#1}{#2}} 
\newcommand{\defeq}{\stackrel{\rm def}{=}}

\numberwithin{equation}{section}

\begin{document}
\setstcolor{red}

\title{An Algorithm for Quadratic $\ell_1$-Regularized Optimization with a Flexible Active-Set Strategy}

\author{Richard H. Byrd$\dagger$\thanks{$\dagger$Department of Computer Science, University of Colorado, Boulder, CO, USA. This author was supported by National Science Foundation grant DMS-1216554 and Department of Energy grant DE-SC0001774.} Jorge Nocedal$\ddagger$ \thanks{$\ddagger$Department of Industrial Engineering and Management Sciences, Northwestern University, Evanston, IL, USA. This material is based upon work supported by the U.S. Department of Energy Office of Science, Office of Advanced Scientific Computing Research, Applied Mathematics program under Award Number FG02-87ER25047. These authors were also supported by  National Science Foundation grant DMS-0810213.} Stefan Solntsev$\ddagger$ }

  \providecommand{\keywords}[1]{\textbf{\textit{Index terms---}} #1}

\maketitle 
\begin{abstract}
	We present an active-set method for minimizing an objective that is the sum of a convex quadratic and $\ell_1$ regularization term. Unlike two-phase methods that combine a first-order active set identification step and a subspace phase consisting of a \emph{cycle} of conjugate gradient (CG) iterations, the method presented here has the flexibility of computing  a first-order proximal gradient step or a subspace CG step at each iteration. The decision of which type of step to perform is based on the relative magnitudes of some scaled components of the minimum norm subgradient of the objective function. 
The paper establishes global rates of convergence, as well as work complexity estimates for two variants of our approach, which we call the iiCG method. Numerical results illustrating the behavior of the method on 
a variety of test problems are presented. 
\end{abstract}

\section{Introduction}

\label{intro} In this paper, we present an active-set method for the solution of the regularized quadratic problem 
\begin{equation}
	\label{prob} 
	\min_{x \in \mathbb{R}^n} F(x) 
	\defeq \fraction{1}{2}x^T A x - b^T x +\tau{ \| x \| }_1, 
\end{equation}
where $A$ is a symmetric positive semi-definite matrix and $\tau \geq 0$ is a regularization parameter. The motivation for this work stems from the numerous applications in signal processing, machine learning, and statistics that require the solution of problem~\eqref{prob}; see e.g.~\cite{4839045,hastie-book,sra2011optimization} and the references therein. 

Although non-differentiable, the quadratic-$\ell_1$ problem~\eqref{prob} has a simple structure that can be exploited effectively in the design of algorithms, and in their analysis. Our focus in this paper is on methods that incorporate second-order information about the objective function $F(x)$ as an integral part of the iteration. A salient feature of our method is the flexibility of switching between two types of steps: a) first-order steps that improve the active-set prediction; b) subspace steps that explore the current active set through an inner conjugate gradient  (CG) iteration.  The choice between these steps is controlled by the \emph{gradient balance condition} that compares the norm of the free (non-zero) components of the minimum norm subgradient of $F$ with the norm of the components corresponding to the zero variables (appropriately scaled). This condition is motivated by the work of Dostal and Schoeberl~\cite{dostal_minimizing_2005} on the solution of bound constrained problems, but in extending the idea to the quadratic-$\ell_1 $ problem~\eqref{prob}, we deviate from their approach in a significant way. 

We present two variants of our approach that differ in the active set identification step. One variant employs the iterative soft-thresholding algorithm, ISTA~\cite{Daubechies:04,ista,sparsa}, while the other computes an ISTA step on the subspace of zero variables. Our numerical tests show that the two algorithms perform efficiently compared to state-of-the-art codes. We provide global rates of convergence as well as work-complexity estimates that bound the total amount of computation needed to achieve an $\epsilon$-accurate solution. We refer to our approach as the ``interleaved ISTA-CG method'', or iiCG.

The quadratic-$\ell_1$ problem~\eqref{prob} has received considerable attention in the literature, and a variety of first and second order methods have been proposed for solving it. Most prominent are variants of the ISTA algorithm and its accelerated versions~\cite{nesterov83,fista,becker2011templates}, which have extensive theory and are popular in practice. The TFOCS package~\cite{becker2011templates} provides five first-order methods based on proximal gradient iterations that enjoy optimal complexity bounds; i.e., they achieve $\epsilon$ accuracy in at most $O(1/\sqrt{\epsilon})$ iterations. One of these methods, N83, is tested in our numerical experiments. Other first-order methods for problem~\eqref{prob} include LARS~\cite{efron_least_2004}, coordinate descent~\cite{friedman_regularization_2010}, a fixed point continuation method~\cite{hale_fixed-point}, and a gradient projection method~\cite{FigueiredoNowakWright:07}. 

Schmidt~\cite{schmidt_thesis} proposed several scaled sub-gradient methods, which can be viewed as extensions, to quadratic-$\ell_1$ problem~\eqref{prob}, of a projected quasi-Newton method~\cite{andrew2007scalable}, an active-set method~\cite{perkins2003}, and a two-metric projection method~\cite{gafni1984} for bound constrained optimization. He compares these methods with some first-order methods such as GSPR~\cite{FigueiredoNowakWright:07} and SPARSA~\cite{sparsa}. We include the best-performing method (PSSgb) in our numerical tests.

Other second order methods have been proposed as well; they compute a step by minimizing a local quadratic model of $F$. Some of these algorithms transform problem~\eqref{prob} into a smooth bound constrained quadratic programming problem and apply an interior point procedure~\cite{kim_interior-point_2007} or a second order gradient projection algorithm~\cite{schmidt_fast,FigueiredoNowakWright:07}. Methods that are closer in spirit to our approach include FPC\_AS~\cite{wen2010fast}, orthant-based Newton-CG methods~\cite{andrew2007scalable,figi,olsen2012newton}, and the semi-smooth Newton method in~\cite{milzareksemismooth}. Our method differs from all these approaches in the adaptive step-by-step nature of the algorithm, where a different kind of step can be invoked at every iteration. This gives the algorithm the flexibility to adapt itself to the characteristics of the problem to be solved, as we discuss in our numerical tests.

The paper is organized in six sections. In Section~\ref{algo} we motivate our approach and describe the first algorithm. Section~\ref{variants} presents the second algorithm. A convergence analysis and a work complexity estimate of the two variants is given in Section~\ref{global}.  Section~\ref{numer} describes implementation details, such as the use of a line search in the identification phase,  and  numerical results. The contributions of the paper are summarized in Section~\ref{remarks}.

\section{The First Algorithm: iiCG-1}\label{algo} \setcounter{equation}{0}
Let us define
\begin{equation*}
          f(x) \defeq \fraction{1}{2}x^T A x - b^T x, \quad\mbox{and} \quad 
           g(x)  \defeq \nabla f(x) = Ax-b,
 \end{equation*}
so that 
\[
        F(x)= f(x) +\tau{ \| x \|}_1.
\]
Throughout the paper, we assume that $\tau$ is fixed and has been chosen to achieve some desirable properties of the solution of the problem.

The algorithm starts by computing a first-order active-set identification step. Then, a subspace minimization step is computed over the space of free variables (i.e. the non-zero variables given by the first-order step) using  the conjugate gradient (CG) method. After each CG iteration, the algorithm determines whether to continue the subspace minimization or perform a first-order step. This decision is based on the so-called gradient balance condition that we now describe. 

The iterative soft-thresholding (ISTA)~\cite{ista,Daubechies:04} algorithm generates iterates as follows: 
\begin{equation}  \label{ista}
	x^{k+1}  = \arg \min_y \,m^k(y) \defeq \arg \min_{y}  f(x^k) + {(y-x^k)}^T g(x^k) +\frac{1}{2 \alpha }{\| y-x^k\| }^2 + \tau{ \| y \| }_1,
\end{equation}
where $\alpha>0$ a steplength parameter.
Since $m^k$ is a separable function, we can write it as 

\begin{align}
	\label{origista} x^{k+1} &=x^k - \alpha \omega(x^k) - \alpha \psi(x^k ),
\end{align}
where, for $i=1, \ldots, n$,
\begin{align}
	\omega{(x^k)}_i &\defeq \left\lbrace 
	\begin{array}{lll}
		0 & \mbox{ if } & x^k_i \neq 0 \\
		\frac{1}{\alpha} (x^k_i-\arg \min_{y_i} \,m^k(y))  &  \mbox{ if } & x^k_i  = 0  \\ 
	\end{array}
	\right\rbrace \nonumber \\
	&= \left\lbrace
	\begin{array}{lll} 
		0 & \mbox{ if } &x^k_i \neq 0 \\
		0 & \mbox{ if } &x^k_i  = 0   \mbox{ and } | g_i(x^k)| \leq \tau \\
		g_i(x^k) - \tau \sgn ( g_i(x^k)) & \mbox{ if } &x^k_i  = 0   \mbox{ and } | g_i(x^k)| > \tau\\
	\end{array}
	\right\rbrace,\label{minv}\\ \nonumber \\
	\quad \psi{(x^k)}_i &\defeq \left\lbrace  
	\begin{array}{lll}
	    \frac{1}{\alpha} (x^k_i-\arg \min_{y_i} \,m^k(y)  & \mbox{ if } & x^k_i \neq 0 \\
		0 & \mbox{ if } & x^k_i  = 0  \\
	\end{array}
	\right\rbrace \nonumber \\
	& =\left\lbrace  
	\begin{array}{lll}
	    \frac{1}{\alpha} \left( x^k_i-\max \{|x^k_i - \alpha  g_i(x^k) | - \alpha  \tau, 0 \} \sgn (x^k_i - \alpha  g_i(x^k)) \right) & \mbox{ if } & x^k_i \neq 0 \\
		0 & \mbox{ if } & x^k_i  = 0  \\
	\end{array}
	\right\rbrace .
	\label{phitildedef}
\end{align}

Following Dostal and Schoeberl~\cite{dostal_minimizing_2005}, we use the magnitudes of the vectors $\omega$ and $ \psi$ to determine which type of step should be taken. The vector $\omega(x^k)$ consists of the components of the minimum norm subgradient of $F$ corresponding to variables at zero (see~\eqref{vdef}),
and its norm provides a first-order estimate of the expected decrease in the objective resulting from changing those variables. 
Similarly, 
$\|\psi(x^k)\|$ 
provides such an estimate for the free variables, but it is more complex due to the effect of variables changing sign.

Thus, when the magnitude of $\omega(x^k)$ is large, it is an indication that releasing some of the zero variables can produce substantial improvements in the objective. On the other hand, when 
$\|\psi(x^k)\|$
 is larger than $ \| \omega(x^k)\|$, it is an indication that a move in the non-zero variables is more beneficial. 
The algorithm thus monitors the \textit{gradient balance condition} 
\begin{equation}{
	\|\omega(x^k) \|}_2 \leq {\|  \psi(x^k ) \|}_2, 
	\label{gamma} 
\end{equation}
which governs the flow of the iteration and distinguishes it from both two-phase methods~\cite{andrew2007scalable,figi,sparsa} and semi-smooth Newton methods~\cite{milzareksemismooth} for problem~\eqref{prob}.

Let us the describe the algorithm in more detail. The first order step is computed by the ISTA iteration~\eqref{origista}; see e.g.~\cite{figi}. The choice of the parameter $\alpha$ is a is discussed below.

The subspace minimization procedure uses the conjugate gradient method to reduce a model of the objective $F(x)$ on the subspace 
\begin{equation*}
	 H= \{ x | \, x_i= 0 , \quad\mbox{for all $i$ such that } x_i^{\rm cg}=0 \} , 
\end{equation*}
where $x^{\rm cg}$ denotes the point at which the CG procedure was started (this point is provided by the ISTA step). 
The conjugate gradient method is applied to a smooth quadratic function $q$ (as in~\cite{wen2010fast}) that equals the objective $F$ on the current orthant defined by $x^{\rm cg}$, i.e., 
\begin{eqnarray}
	\label{qz} q(x;x^{\rm cg}) 
	\defeq \fraction{1}{2}x^T A x + {(-b + \tau \sgn(x^{\rm cg}))}^T x, 
\end{eqnarray}
where we use the convention $\sgn(0)=0$ and the fact that
\begin{equation}
       \|x\|_1 = \sgn{(x)}^T x. 
	    \label{light}
 \end{equation}
Clearly, $ F(x) = q(x;x^{\rm cg}) $ for all $x$ such that $\sgn(x)= \sgn(x^{\rm cg})$. The algorithm applies the projected CG iteration~\cite[chap 16]{mybook} to the problem 
\begin{align}
	\min_x & \ \ q(x;x^{\rm cg}) \label{cgit} \\
	{\rm s.t.} & \ \ x \in H. \nonumber
\end{align}
The gradient of $q$ at $x^k$ on the subspace $H$, is given by $P(g(x^k) + \tau \sgn(x^k))$, where $P$ is the projection onto $H$. It follows that an iteration of the projected CG method is given by 
\begin{align*}
	x^{k+1} &= x^k + \alpha _{\rm cg}d^k, \quad\mbox{with} \quad \alpha _{\rm cg}= \frac{{(r^k)}^T \rho^k}{{(d^k)}^T A d^k}; \\
	r^{k+1} &= r^k + \alpha _{\rm cg} Ad^k; \nonumber \\
	\rho^{k+1} &= P(r^{k+1}); \nonumber \\
	d^{k+1} &= -\rho^{k+1}+\frac{{(r^{k+1})}^T \rho^{k+1}}{{(r^k)}^T \rho^k} d^k, \nonumber 
\end{align*}
where initially $r^k = g(x^k) + \tau \sgn(x^k)$, $\rho^k=P(r^{k})$, $d^k= -\rho^k$. 

The gradient balance condition~\eqref{gamma} is tested after every CG iteration, and if it is not satisfied, the CG loop is terminated. This is a sign that substantial improvements in the objective value can be achieved by releasing some of the zero variables. This loop is also terminated when the CG iteration has crossed orthants and a sufficient reduction in the objective $F$ was not achieved.

A precise description of the method for solving problem~\eqref{prob} is given in Algorithm iiCG-1. Here and henceforth $\| \cdot \|$ stands for the $\ell_2$ norm, and $v(x)$ denotes the minimum norm subgradient of $F$ at $x$.
\newpage

\bigskip 
\begin{algorithm}
	 \caption{Algorithm iiCG-1} \label{alg1} 
	\begin{algorithmic}
		[1] \REQUIRE $A$, $b$, $\tau$, $x^0$, $c$, and $\alpha $
		 \STATE $k=0$ 
		 \LOOP
		  \STATE $x^{k+1} = x^k - \alpha \omega(x^k) - \alpha \psi(x^k) $ \qquad\qquad\qquad\qquad   \qquad  \qquad  \;  \textit{ISTA step} 
		  \STATE $k=k+1$ 
		  \STATE $r^k = g(x^k) + \tau \sgn(x^k)$, $\rho^k=P(r^{k})$, $d^k= -\rho^k$,  $x^{\rm cg} = x^k$ 
		  \LOOP 
		  \IF {$ \|\omega(x^k) \| > \|  \psi(x^k ) \|$} 
		  \STATE \textbf{break} 
		  \ENDIF \STATE $x^{k+1} = x^k + \frac{(r^k)^T \rho^k}{{(d^k)}^T Ad^k} d^k$ \quad\qquad\quad\qquad\qquad\qquad\quad\quad \; \; \; \; \textit{CG step} 
		  \STATE $r^{k+1} = r^k + \frac{{(r^k)}^T \rho^k}{{(d^k)}^T Ad^k} Ad^k$ 
		  \IF {$\sgn(x^{k+1}) \neq  \sgn (x^{\rm cg})$ and $F(x^{k+1}) > F(x^{k}) - c {\| v(x^{k}) \|}^2$ }
		  \STATE $x^{k+1}=$ \texttt{cutback}$(x_k, x^{\rm cg},d^k)$ 
		  \STATE $k=k+1$
		   \STATE \textbf{break}
		    \ENDIF 
		    \STATE $\rho^{k+1} = P(r^{k+1})$ 
		    \STATE $d^{k+1} = -\rho^{k+1} +\frac{{(r^{k+1})}^T \rho^{k+1}}{{(r^k)}^T \rho^k} d^k$ 
		    \STATE $k=k+1$ 
		    \ENDLOOP 
		    \ENDLOOP 
	\end{algorithmic}
\end{algorithm}

A common choice for the stepsize is 
$\alpha  = 1/L$
 (where $L$ is the largest eigenvalue of $A$) and is motivated by the convergence analysis in Section~\ref{global}. In practice, precise knowledge of $L$ is not needed; in Section~\ref{numer} we discuss a heuristic way of choosing $\alpha$.

Let us  consider Step 13. It can be beneficial to allow the CG iteration to leave the current orthant, as long as the objective $F$ is reduced sufficiently after every CG step. Inspired by the analysis given in Section~\ref{global} (see Lemma~\ref{lem:cgstep}), we require that 
\begin{equation} 
	F(x^{k+1}) \leq F(x^{k}) - c {\| v(x^{k}) \|}^2, 
	\label{sufd}
\end{equation}
for some $c \geq 0$, where $v(x^k)$ is the minimum norm subgradient of $F$ at $x^k$. The CG iteration is thus terminated as follows. If $x^{k+1}$ is the first CG iterate that leaves the orthant, then we either accept it, if it produces the sufficient decrease~\eqref{sufd} in $F$, or we cut it back to the boundary of the current orthant. On the other hand, if both $x^{k+1}$ and $x^k$ lie outside the current orthant and if sufficient decrease is not obtained at $x^{k+1}$, then the algorithm reverts to $x^k$. This procedure is given as follows:

\bigskip 
\textbf{cutback}$(x_k, x^{\rm cg},d^k)$ 
\begin{algorithmic}
 \IF {$\sgn(x^{k}) = \sgn(x^{\rm cg})$ }
  \STATE $\alpha _b = \arg \max_{\alpha _b } \{\alpha _b : \sgn(x^k + \alpha _b d^k ) = \sgn(x^{\rm cg}) \}$
   \STATE $x^{k+1}=x^k + \alpha _b d^k$ 
   \ELSE 
   \STATE $x^{k+1}=x^k$ 
   \ENDIF 
\end{algorithmic}

\medskip 
\noindent
\newpage
One of the main benefits of allowing the CG iteration to move across orthant boundaries is that this strategy prevents the generation of unnecessarily short subspace steps. In addition, if the quadratic model~\eqref{qz} does not change much as orthants change (for example, when $\tau$ is small), CG steps can be beneficial in spite of the fact that they are based on information from another orthant.

Note that in algorithm iiCG-1 the index $k$ may be incremented multiple times during every outer iteration loop. While it is possible to express the same algorithmic logic in a more conventional way, with a single $k$ increment in each outer iteration, our description emphasizes that every inner CG step and ISTA step require approximately the same amount of computational effort, which is dominated by a matrix-vector product.

The algorithm of Dostal and Schoeberl includes a step along the direction $- \omega(x^k)$; its goal is release variables when the gradient balance condition does not hold. We have dispensed with this step, as we have observed that it is more effective to release variables through the ISTA iteration.

\section{The Second Algorithm: iiCG-2} \label{variants} \setcounter{equation}{0}

This method is motivated by the observation that the ISTA step~\eqref{origista} often releases too many zero variables. To prevent this, we replace it (under certain conditions) by a \emph{subspace ISTA step} given by
\begin{align}
	\label{distast} 
	x^{k+1} = x^k - \alpha \psi(x^k ).
\end{align}
 Here, zero variables are kept fixed and the rest of the variables are updated by the ISTA iteration; see definition~\eqref{phitildedef}. Thus,~\eqref{distast} refines the estimate of the active set without releasing any variables.
 
 The subspace ISTA step~\eqref{distast} is performed only when the gradient balance condition~\eqref{gamma} is satisfied. If~\eqref{gamma} is not satisfied, releasing some of the zero variables may be beneficial, and the full-space ISTA step~\eqref{origista} is taken to allow this.
 The freedom to choose among  two types of active-set prediction steps provides the algorithm with a powerful active set identification mechanism; see the results in Section~\ref{numer}. The choice of the steplength $\alpha$ in~\eqref{distast} is discussed in Section~\ref{numer}. A detailed description this method is given in algorithm iiCG-2.

\begin{algorithm}
	 \caption{Algorithm iiCG-2} \label{alg2} 
	\begin{algorithmic}
		[1] \REQUIRE $A$, $b$, $\tau$, $x^0$, $c$, and $\alpha $
		\STATE $k=0$
		\LOOP 
		\IF {$ \|\omega(x^k) \|^2 \leq \|\psi(x^k)\|^2$}
			 \STATE$x^{k+1} = x^k - \alpha \psi(x^k ) $\quad\qquad\qquad\qquad \qquad\qquad\qquad\qquad\;  \;  \textit{Subspace ISTA step} 
		   	\ELSE 
		  \STATE $x^{k+1} = x^k - \alpha \omega(x^k) - \alpha \psi(x^k) $ \qquad\qquad\qquad\qquad   \qquad  \quad \; \textit{ISTA step} 
		     \ENDIF 
		     \STATE $k=k+1$ 
			 \STATE $r^k = g(x^k) + \tau \sgn(x^k)$, $\rho^k=P(r^{k})$, $d^k=-\rho^k$, $x^{\rm cg} = x^k$
		      \LOOP 
		\IF {$ \|\omega(x^k) \|^2 > \|\psi(x^k)\|^2$}
		       \STATE \textbf{break} 
		       \ENDIF 
		       \STATE $x^{k+1} = x^k + \frac{{(r^k)}^T \rho^k}{{(d^k)}^T Ad^k} d^k$ \qquad \qquad \qquad \qquad \qquad \qquad \qquad \; \textit{CG step} 
		       \STATE $r^{k+1} = r^k + \frac{{(r^k)}^T \rho^k}{{(d^k)}^T Ad^k} Ad^k$ 
		       \IF {$\sgn(x^{k+1}) \neq  \sgn (x^{\rm cg})$ and $F(x^{k+1}) > F(x^{k}) - c {\| v(x^{k}) \|}^2$ } 
		       \STATE $x^{k+1}=$ \texttt{cutback}$(x_k, x^{\rm cg},d^k)$ 
		          \STATE $k=k+1$ 
		          \STATE \textbf{break}
		           \ENDIF
		            \STATE $\rho^{k+1} = P(r^{k+1})$
		             \STATE $d^{k+1} = -\rho^{k+1} +\frac{{(r^{k+1})}^T \rho^{k+1}}{{(r^k)}^T \rho^k} d^k$ 
		             \STATE $k=k+1$
		              \ENDLOOP 
		              \ENDLOOP 
	\end{algorithmic}
\end{algorithm}

\section{Convergence Analysis} \label{global} \setcounter{equation}{0}

We establish global convergence for algorithms iiCG-1 and iiCG-2 by showing that their constitutive steps, ISTA, subspace ISTA and CG, provide sufficient decrease in the objective function. We also show a global 2-step Q-linear rate of convergence, and based on the fact that the number of cut CG steps cannot exceed half of the total number of steps, we establish a complexity result. In addition, we establish finite active set identification and termination for the two iiCG methods.

In this section, we assume that $A$ is nonsingular, and denote its smallest and largest eigenvalues by $\lambda$ and $L$, respectively. Thus, for any $x \in \mathbb{R}^n$, 
\begin{equation*}
	\lambda \|x\|^2 \leq x^T A x \leq L \|x\|^2.
\end{equation*} 
We denote the minimizer of $F$ by $x^*$, and in the rest of this section we use the abbreviations
\begin{equation}  
   v=v(x^k), \quad \omega = \omega(x^k), \quad \phi = \phi(x^k), \quad  \psi= \psi(x^k) , \quad g=g(x^k),
    \label{stereo}
 \end{equation}
 when convenient.
We start by demonstrating a Q-linear decrease in the objective $F$ for every ISTA step. 
\begin{lem}
	\label{lem:ista} The ISTA step, 
	\begin{align*}
		x^{k+1} = \arg \min_y m^k(y) = x^k - \alpha \omega(x^k) - \alpha \psi(x^k ), 
	\end{align*}
	with $0< \alpha  \leq 1/L$, satisfies 
	\begin{align*}
		F(x^{k+1}) - F(x^* ) \leq (1-\lambda \alpha  )(F(x^k)-F(x^* ) ). 
	\end{align*}
\end{lem}
\begin{proof} 
     Since $1/\alpha \geq L$, we have that $m^k(y)$ defined in~\eqref{ista} is a majorizing function for $F$ at $x_k$. Therefore,
	\begin{align*}
		F(x^{k+1}) &\leq m^k(x^{k+1}). 
	\end{align*}
	Since $x^{k+1}$ is the minimizer of $m^k$, for any $d \in \mathbb{R}^n$, we have 
	\begin{align}
		F(x^{k+1}) & \leq m^k(x^{k+1}) \nonumber \\
		 & \leq m^k(x^k+\lambda \alpha d) \nonumber \\
		& = F(x^k+ \lambda \alpha  d ) - \fraction{1}{2} {( \lambda\alpha  d)}^T A ( \lambda \alpha  d) + \frac{1}{2\alpha } \| \lambda \alpha  d\| ^2 \nonumber \\
		& \leq F(x^k+\lambda \alpha  d ) + \fraction{1}{2}\lambda^2\alpha (1 -\lambda\alpha ) \| d\| ^2 . \label{appr} 
	\end{align}
	Since $F$ is a strongly convex function with parameter $\lambda$, it satisfies.
	\begin{equation}
		 F(tx+(1-t)y) \leq tF(x) +(1-t)F(y) - \fraction{1}{2} \lambda t (1-t) \|x-y\|^2, 
		\label{strong1}
	\end{equation}
	for any $x, y \in \mathbb{R}^n$ and $t \in [0,1]$, see~\cite[Pages 63-64]{nesterov2004}. Setting $x \leftarrow x^k$, $y \leftarrow x^*$, and $t \leftarrow (1- \lambda \alpha )$ (which is valid because $\lambda \alpha  \in (0, 1]$), inequality~\eqref{strong1} yields
	\begin{equation} 
	      F(x^k+\lambda \alpha  ( x^* -x^k) ) \leq \lambda\alpha F(x^*) +(1-\lambda\alpha )F(x^k) - \fraction{1}{2}
	       \lambda^2 \alpha  (1-\lambda\alpha ) \|x^*-x^k\|^2 . 
		   \label{mage}
	\end{equation}
	Since~\eqref{appr} holds for any $d$, we can set $d= x^* -x^k$. Substituting~\eqref{mage} in~\eqref{appr}, we conclude that 
	\begin{align*}
		F(x^{k+1}) -F(x^* ) &\leq (1-\lambda\alpha  )(F(x^k)-F(x^* ) ) . 
	\end{align*}
\end{proof}

We now establish a similar result for the subspace ISTA step~\eqref{distast}, under the conditions for which it is invoked, namely when the gradient balance condition is satisfied.

		\begin{lem}
			\label{lem:subista} If 
			$\|\omega(x^k)\| \leq \|\psi(x^k)\|$
			the subspace ISTA step, 
			\begin{eqnarray}
				\label{subistast} x^{k+1} = x^k-\alpha \psi(x^k) , 
			\end{eqnarray}
			with $0< \alpha  \leq 1/L$, satisfies  
			\begin{equation*}
				F(x^{k+1}) - F(x^* ) \leq  (1- \fraction{1}{2} \lambda \alpha  )(F(x^k)-F(x^* ) ). 
			\end{equation*}
		\end{lem}
		
				\begin{proof}

By the definitions~\eqref{minv} and~\eqref{phitildedef}, we have that  $\omega^T \psi=0$, and hence $\omega^T(x^{k+1}-x^{k})=0$ (Recall the abbreviations \eqref{stereo}). Given the pattern of zeros in $x^k$, $\omega$ and $\psi$, we have also that $\| x^{k+1} - \alpha \omega\|_1 =  \| x^{k+1} \|_1+\alpha\|  \omega\|_1$. Using these two observations, we have for the full ISTA step,
		\begin{align}
			\nonumber
		&m^k(x^k-\alpha \omega -\alpha \psi) \\
			\nonumber
		 =& \, f(x^k) + {(-\alpha \omega -\alpha \psi)}^T g +\frac{1}{2 \alpha }\| -\alpha \omega -\alpha \psi\| ^2 + \tau \|x^k-\alpha \omega -\alpha \psi\|_1  \\
			\nonumber
		       =& \, f(x^k) + {( x^{k+1} - \alpha \omega-x^k)}^T g +\frac{1}{2 \alpha }\| x^{k+1} - \alpha \omega-x^k\| ^2 + \tau \| x^{k+1} - \alpha \omega\|_1 \\
			\nonumber
		       =& \, f(x^k) + {( x^{k+1} -x^k)}^T g +\frac{1}{2 \alpha }\| x^{k+1} -x^k\| ^2 + \tau \| x^{k+1} \|_1  - \alpha \omega^T g  +\frac{1}{2 \alpha }\|  \alpha \omega\| ^2+ \tau \| \alpha \omega\|_1 \\
			\nonumber
			   =& \, m^k(x^{k+1})  - \alpha \omega^T g  +\frac{1}{2 \alpha }\|  \alpha \omega\| ^2+ \tau \alpha \omega^T \sgn (\alpha \omega ) \\
			   \label{comp1}
		       =& \, m^k(x^{k+1})  -\frac{\alpha}{2  }\| \omega\| ^2,  
		\end{align}
where the last equality follows from the fact that, by~\eqref{minv}, for all $i$ s.t. $\omega_i \neq 0$, we have $\sgn(g_i)= \sgn(\omega_i)$, and that $w^Tw = w^Tg - \tau \omega^T \sgn(g)$.

				For the subspace ISTA step~\eqref{subistast} we have
				\begin{align}  
					\nonumber
				m^k(x^{k+1}) & = f(x^k) + {(x^{k+1}-x^k)}^T g +\frac{1}{2 \alpha }\| x^{k+1}-x^k\| ^2 + \tau \|x^{k+1}\|_1  \\
					\nonumber
				           & = f(x^k) - \alpha  \psi^T g +\frac{1}{2 \alpha }\|\alpha  \psi\| ^2 + \tau \|x^{k+1}\|_1  \\
				           & = m^k(x^k) - \alpha  \psi^T g +\frac{\alpha}{2  }\|  \psi\| ^2 + \tau \|x^{k+1}\|_1 - \tau \|x^k\|_1.  
						   \label{more}
				\end{align}
We examine the second term on the right hand side. For $j$ such that $x_j^k=0$ we have $\psi_j(x^k)=0$. On the other hand, for $x_j^k \neq 0$ definitions~\eqref{origista} and~\eqref{ista} yield
\begin{equation*}
      x_i^{k+1} = x_i^k + \alpha \psi_i(x^k)= \arg \min_{y_i} f(x^k) + (y_i - x_i^k)g_i(x^k) + \fraction{1}{2\alpha} 
      {(y_i - x_i^k)}^2       + \tau |y_i|.
 \end{equation*}      
				In examining the optimality conditions of this problem there are two cases. If $x_j^{k+1}\neq 0$, we obtain $-g_j=   \psi_j + \tau \sgn(x^{k+1}_j)$. On the other hand, if $x^{k+1}_j =0$ we  have $ -g_j=  \psi_j  +\tau \sigma_j$     for some $\sigma_j \in [-1,1]$.   Substituting for $g_j$ in~\eqref{more}, and using~\eqref{light} gives
		                 \begin{align}
					\nonumber
			    m^k(x^{k+1}) = & \, m^k(x^k) +\sum_{j:x_j^{k+1} \neq 0} \tau \alpha \psi_j \sgn(x^{k+1})_j +\sum_{j:x_j^{k+1} = 0} \tau \alpha \psi_j \sigma_j -\frac{\alpha}{2} \| \psi \|^2  + \tau \|x^{k+1}\|_1 - \tau \|x^k\|_1  \\
					\nonumber
				            = & \, m^k(x^k) +\sum_{j:x_j^{k+1} \neq 0} \tau \alpha \psi_j \sgn(x^{k+1})_j +\sum_{j:x_j^{k+1} = 0} \tau \alpha \psi_j \sigma_j -\frac{\alpha}{2} \| \psi \|^2  \\ \nonumber & + \tau {(x^k - \alpha \psi)}^T\sgn(x^{k+1}) - \tau \|x^k\|_1  \\
					\nonumber
				           = & \, m^k(x^k) +\tau \sum_{j:x_j^{k+1} \neq 0} \left[x^k_j \sgn (x^{k+1})_j -|x^k_j| \right]+\tau \sum_{j:x_j^{k+1} = 0} \left[ x^k_j \sigma_j- |x^k_j| \right]  -\frac{\alpha}{2} \| \psi \|^2  \\
				            \leq & \, m^k(x^k)  -\frac{\alpha}{2} \| \psi \|^2.
					\label{comp2}
				\end{align}

Using in turn~\eqref{comp1}, the gradient balance condition $\|\psi\| \geq \|\omega\|$ and~\eqref{comp2} we have
	\begin{align*}
	m^k(x^k)-m^k(x^k-\alpha \omega -\alpha \psi) &= m^k(x^k)-m^k(x^{k+1}) +m^k(x^{k+1}) -m^k(x^k-\alpha \omega -\alpha \psi)       \\
	              &=m^k(x^k)-m^k(x^{k+1}) + \frac{\alpha}{2} \|\omega\|^2 \\
	&\leq m^k(x^k) -m^k(x^{k+1}) + \frac{\alpha}{2} \| \psi \|^2 \\
	& \leq 2 [m^k(x^k) -m^k(x^{k+1}) ].
	\end{align*}
Since $m^k(x^k)= F(x^k)$, this relation yields 
\[
    m^k(x^{k+1}) - F(x^k) \leq \fraction{1}{2} [m^k(x^k-\alpha \omega -\alpha \psi) - F(x^k) ].
\]
Using this bound, the fact that  $m^k$ is a majorizing function of $F$, and that $(x^k-\alpha \omega -\alpha \psi)$ is a minimizer of $m^k$, we have that for any  $d \in {R}^n$, 
		\begin{align}
		F(x^{k+1}) -F(x^k) &  \leq  \frac{1}{2} [ m^k(x^k + \lambda\alpha  d) -F(x^k) ] \nonumber \\
		& = \frac{1}{2} [ f(x^k) +g^T (\lambda\alpha  d)+ \frac{1}{2 \alpha } \| \lambda \alpha d\| ^2 
		                 +\tau \| x^k+\lambda \alpha d \|_1 -F(x^k) ] \nonumber  \\
		& = \frac{1}{2} [ F(x^k+ \lambda \alpha  d ) -F(x^k) - \fraction{1}{2} {( \lambda\alpha  d)}^T A ( \lambda \alpha  d) + \frac{1}{2\alpha } \| \lambda \alpha  d\| ^2 ] \nonumber \\
		& \leq \frac{1}{2} [F(x^k+\lambda \alpha  d ) - F(x^k) + \fraction{1}{2}\lambda^2\alpha (1 -\lambda\alpha ) \| d\| ^2 ] . \label{grey} 
		\end{align}

		Since  $F$ is a strongly convex function with parameter $\lambda$, it  satisfies 
				\begin{equation}
					F(x+t d) \leq F(x) +t(F(x+d) - F(x)) - \fraction{1}{2} \lambda (1-t) t \|d\|^2, 
				\label{strong2} 
				\end{equation}
				for any $x, d \in {R}^n$ and $t \in [0,1]$. Setting $x \leftarrow x^k$,  and $t \leftarrow  \lambda \alpha $ (which is valid because $\lambda \alpha  \in (0, 1]$), inequality~\eqref{strong2} yields
					\[ F(x^k+\lambda \alpha  d ) \leq \lambda\alpha F(x^k) +(1-\lambda\alpha )F(x^k + d) - \fraction{1}{2} \lambda^2 \alpha  (1-\lambda\alpha ) \|d\|^2 . \]
					 Substituting this inequality in~\eqref{grey}, and setting $d= x^* -x^k$ we conclude that 
					\begin{eqnarray*}
						F(x^{k+1}) -F(x^k ) &\leq \frac{1}{2}  (\lambda\alpha  )(F(x^k +d) -F(x^k ) )  \\ 
						                    &\leq \frac{1}{2}  (\lambda\alpha  )(F(x^* ) -F(x^k ) ) ,
					\end{eqnarray*}
				which implies the result.
				\end{proof}
				
Next, we analyze the conjugate gradient step. To do so, we first  introduce some notation and a technical lemma. The subgradient of $F(x)$ \cite{nesterov2004}, of least norm, is given by
                
				\begin{equation}
					v_i(x) = \left\lbrace
					\begin{array}{lll}
						g_i(x) + \tau \sgn (x_i) & \mbox{ if } & x_i \not = 0 \\
						0 & \mbox{ if } & x_i = 0 \mbox{ and } | g_i(x)| \leq \tau \\
						g_i(x) - \tau \sgn ( g_i(x)) & \mbox{ if } & x_i = 0 \mbox{ and }| g_i(x)| > \tau \\
					\end{array}
					\right\rbrace \mbox {for } \ i=1, \ldots, n.
					\label{vdef} 
				\end{equation}
				Recalling \eqref{minv}, we can write $v(x) = \omega(x) + \phi(x)$, where 

				\begin{equation}\phi_i(x) \defeq \left\lbrace
					\begin{array}{lll}
						g_i(x) + \tau \sgn (x_i) & \mbox{ if } & x_i \not = 0 \\
						0 & \mbox{ if } & x_i = 0 \\
					\end{array} 
					\right\rbrace \mbox {for } \ i=1, \ldots, n.
					\label{phidef} 
				\end{equation}
				
The following result establishes a relationship between $\psi $ and $\phi$. 

				\begin{lem}
					\label{lem:philess} 
				For all $x^k$, we have that $ \| \psi (x^k)  \| \leq \| \phi (x^k)\|$.
				\end{lem}
				\begin{proof} 

				We  show that $| \psi_j| \leq |\phi_j| $ for each  $j$ such that $x_j \neq 0$ (the other components of the two vectors are zero).  If $\phi_j=\psi_j$ the result holds, so assume  $\phi_j \neq \psi_j$. For such $j$, we have that $x_j-\alpha \psi_j$ minimizes $m_j^k(y)$, which by~\eqref{ista} is given by
				\[
				m_j^k(y) = f(x^k)+g_j(y-x^k_j) + \fraction{1}{2\alpha} {(y-x^k_j)}^2  +|y|, \quad y \in \mathbb{R}.
				\]
On the other hand, $x_j-\alpha \phi_j$ minimizes 

				\begin{equation*}
				\bar{m}_j^k(y) \defeq f(x^k)+g_j(y-x^k_j) +\fraction{1}{2\alpha}{(y-x^k_j)}^2  + \sgn(x^k_j)  y.
				\end{equation*}
				Thus $  m_j^k(x^k_j -\alpha \psi_j) < m_j^k(x^k_j -\alpha \phi_j)$ and $\bar{m}_j^k (x^k_j -\alpha \psi_j) > \bar{m}_j^k(x^k_j -\alpha \phi_j)$, where the inequalities  
				 are strict since both functions are strictly convex. Subtracting these inequalities, we have 
				\begin{equation*}
				m_j^k(x^k_j -\alpha \psi_j) - \bar{m}_j^k(x^k_j -\alpha \psi_j) < m_j^k(x^k_j -\alpha \phi_j) - \bar{m}_j^k(x^k -\alpha \phi).
				\end{equation*}
				Therefore,
				\begin{equation}
				0 \leq |x^k_j -\alpha \psi_j| -\sgn(x^k_j) (x^k_j -\alpha \psi_j)  < |x^k_j -\alpha \phi_j|- \sgn(x^k_j) (x^k_j -\alpha \phi_j)  ,
				\label{phildeproof}
				\end{equation} 
				showing that the right hand side is positive, which implies that
				\begin{equation}
				\sgn(x^k_j)  \neq \sgn(x^k_j-\alpha \phi_j) .
				\label{sgnsnotequal}
				\end{equation} 
                Now note that, since $m^k_j$ and $\bar{m}^k_j$ coincide in a neighborhood of $x^k_j \neq 0$, $\phi_j$ and $\psi_j$ have the same sign.
				We then consider two cases. 
				If $ \sgn(x^k_j)  = \sgn(x^k_j-\alpha \psi_j)$, then by~\eqref{sgnsnotequal}, the displacement $-\alpha \psi_j$ must be shorter than $-\alpha \phi_j$; therefore, $|\psi_j| < |\phi_j|$.
								On the other hand, if $\sgn(x^k_j)  \neq \sgn(x^k_j-\alpha \psi_j)$, the left side of~\eqref{phildeproof} is $ 2 |x^k -\alpha \psi_j|$.	By~\eqref{sgnsnotequal}, the right side of~\eqref{phildeproof} is $2 |x^k_j -\alpha \phi_j|$.
				Thus $2 |x^k_j -\alpha \psi_j| < 2 |x^k_j -\alpha \phi_j|$. Since the displacement $-\alpha \phi_j$ produces a point farther from zero than the displacement $-\alpha \psi_j$, and since both displacements produce points with signs different from $\sgn(x^k_j)$, we have that $x^k_j -\alpha \phi_j$ is farther from $x^k_j$ than $x^k_j -\alpha \psi_j$. Therefore, $|\psi_j| < |\phi_j|$.

				\end{proof}

			We can now show that a conjugate gradient step also guarantees sufficient decrease in the objective function, provided it is not truncated, i.e. that the {\tt cutback} procedure is not invoked.
			\begin{lem}
				\label{lem:cgstep} If Algorithms iiCG-1 or iiCG-2 take a full conjugate gradient step from $x^k$ to $x^{k+1}$, then 
				\begin{equation}
					 \label{decreasev}
				       F(x^{k+1}) \leq F(x^k) - \beta \| v(x^k) \|^2,
				\end{equation}
				where $\beta = \min \{ c, 1/8L \}$, where $c$ is defined in~\eqref{sufd}.
			\end{lem}
			\begin{proof}
				If $\sgn(x^{k+1}) \neq  \sgn (x^{\rm cg})$, then CG steps are accepted only if condition~\eqref{sufd} is true, and hence \eqref{decreasev} is satisfied. Otherwise $\sgn(x^{k+1})= \sgn (x^{\rm cg})$, and $F(x^{k+1}) = q(x^{k+1} )$. Since we assume $x^{k+1}$ is given by a full CG step, it follows that $x^{k+1}$ is the result of a sequence of projected CG steps on problem~\eqref{cgit}, starting at $x^{cg}$. It is well known that a CG iterate $x^{k+1}$ is a global minimizer of $q(\cdot \,;x^{\rm cg})$ in the subspace $S= {\rm span }\{d^k, d^{k-1}, \ldots \},$ i.e.,
				\[ q(x^{k+1}) = \min \{ q(x^k+y) : y \in S \}. \]
				It is also known that $P(\nabla q(x^k;x^{\rm cg})) \in S$, see~\cite[Theorem 5.3]{mybook}, and it follows from~\eqref{phidef} that $P(\nabla q(x^k;x^{\rm cg}))= \phi$. Thus,
	
				\[ F(x^{k+1}) = q(x^{k+1} )\leq q(x^k - \zeta  \phi), \]
				for any $\zeta $. Let us choose
				\[ \zeta  = \frac{\phi^T \phi}{\phi^T A \phi}. \]
				Recalling~\eqref{qz}, we have 
				\begin{align*}
					F(x^{k+1}) & \leq \fraction{1}{2} (x^k - \zeta  \phi)^T A (x^k - \zeta  \phi) + (-b +\tau \sgn(x^{\rm cg}))^T(x^k - \zeta  \phi ) \\
					& = F(x^k) + \frac{1}{2}\zeta  \phi^T A \zeta  \phi - \zeta  \phi^T A x^k + (-b +\tau \sgn(x^{\rm cg}))^T(- \zeta  \phi ) \\
					& = F(x^k) + \frac{\zeta }{2} \| \phi \|^2 - \zeta  \phi^T (Ax^k - b + \tau \sgn(x^{\rm cg})). 
				\end{align*}
				By~\eqref{phidef}, for $i$ such that $x_i^k=0$ we have $\phi_i=0$, and for $x_i^k \neq 0$ we have that $\phi_i = (Ax^k - b + \tau \sgn(x^{\rm cg}))_i$. Therefore, 
				\begin{align*}
					F(x^{k+1}) & = F(x^k) - \frac{\zeta }{2} \| \phi \|^2 \\
					& \leq F(x^k) -\frac{1}{2 L} \| \phi \|^2. 
				\end{align*}
	
				Since the step is only taken when condition~\eqref{gamma} is true, we have from Lemma~\ref{lem:philess} that
				\[ \|v\| \leq \|\omega\| + \|\phi\| \leq \|\psi\| + \|\phi\| \leq 2\|\phi\|. \]
				Therefore, we have that the following bound holds after one CG iteration,
				\begin{equation} 
				       F(x^{k+1}) \leq F(x^k) - \frac{1}{8 L} \| v \|^2 .
				\end{equation}
				This implies \eqref{decreasev} by definition of $\beta$.
			\end{proof}

			We can now establish a 2-step $Q$-linear convergence result for both algorithms by combining the properties of their constitutive steps
					\begin{thm}
						\label{thm:qlinear} Suppose that the stepsize $\alpha $ in the ISTA step~\eqref{origista} and the subspace ISTA step~\eqref{subistast} satisfy $\fraction{1}{8L} \leq \alpha  \leq \fraction{1}{L}$, and let  $\beta = \min \{ c, 1/8L \}$. Then, for the entire sequence $\{x^k\}$ generated by Algorithms iiCG-1 and iiCG-2 we have 
						\begin{equation}
							\label{pre-recursion} F(x^{k+2}) - F(x^*) \leq \left(1-\frac{\lambda \beta }{2} \right) (F(x^k) - F(x^*)),
						\end{equation}
						and thus $\{ x^k\} \rightarrow x^* $. 
					\end{thm}
					\begin{proof}
				By Lemma~\ref{lem:ista} and the lower bound on $\alpha $, we have that the ISTA step satisfies
				\begin{equation}
					\label{decreaseeee}
					 F(x^{k+1}) - F(x^*) \leq (1-\fraction{\lambda}{16 L}) (F(x^k) - F(x^*)) . 
					 \end{equation}
					 By Lemma~\ref{lem:subista}, and the lower bound on $\alpha$, we have that the subspace ISTA step also satisfies~\eqref{decreaseeee}. By definition of $\beta$, relation \eqref{decreaseeee} implies \eqref{pre-recursion}.
		 
					 Now a that full CG steps provide the decrease~\eqref{decreasev}. Convexity of $F$ shows that
					\[ F(x^k)-F(x^*) \leq -v^T(x^*-x^k) \leq \| v \| \|x^*-x^k \| , \]
					which combined with~\eqref{decreasev} gives
					\[ F(x^k) -F(x^{k+1}) \geq \beta \frac{(F(x^k)-F(x^*))^2}{ \|x^*-x^k \|^2}. \]
					Furthermore, since $F$ is strongly convex, it satisfies
					\[ F(x^k) - F(x^*) \geq \fraction{\lambda}{2} \| x^k-x^*\|^2, \]
					see~\cite[pp. 63-64]{nesterov2004}. Using this bound we conclude that full CG steps satisfy~\eqref{pre-recursion}.

				 Let us assume now that all CG steps terminated by the {\tt cutback} procedure; i.e., that the worst case happens. After every such shortened CG step which may not provide sufficient reduction in $F$, the algorithm computes an ISTA step. Therefore, the Q-linear decrease~\eqref{decreaseeee} is guaranteed for every 2 steps, yielding~\eqref{pre-recursion} for all $k$.
	
				Since the algorithms are descent methods, this implies the entire sequence satisfies $F(x^k) \rightarrow F(x^*)$ monotonically. Moreover, since $F$ is strictly convex it follows that $x^k \rightarrow x^*$. 
					\end{proof}  

			The most costly computations in our algorithms are matrix-vector products; the rest of the computations consist of vector operations. Therefore, when establishing bounds on the total amount of computation required to obtain an $\epsilon$-accurate solution, it is appropriate to measure work in terms of matrix-vector products. Since there is a single matrix-vector product in each of the constitutive steps of our methods, a work complexity result can be derived from Theorem~\ref{thm:qlinear}. 
			\begin{cor}
				The number of matrix-vector products required by algorithms iiCG-1 and iiCG-2 to compute an iterate $\hat{x}$ such that 
				\begin{equation}
					\label{withinepsilonobjective} F(\hat{x}) - F(x^*) \leq \epsilon ,
				\end{equation}
				is at most
				\begin{equation}
					\label{kwork} \log \left[ \frac{\epsilon} {F(x^0) - F(x^*) } \right] \Big/ \log \sqrt{1-\fraction{\lambda \beta}{2}} . 
				\end{equation}
				\label{corl}
			\end{cor}
			\begin{proof}
				By~\eqref{pre-recursion}, the condition~\eqref{withinepsilonobjective}, with $\hat{x}= x^{k+2}$, will be satisfied by an integer $k$ such that
				\[ \left(1-\fraction{\lambda \beta }{2}\right)^{\frac{k}{2}} (F(x^0) - F(x^*)) \leq \epsilon. \]
				We obtain~\eqref{kwork} by solving for $k$. 
			\end{proof}

From the point of view of complexity, choosing $c=1/8L$ is best, but in practice we have found it more effective to use a very small value of $c$ since the strength of the conjugate gradient method  is sometimes observed only after a few iterations are computed outside the orthant. More generally, Corollary~\ref{corl} represents worst-case analysis, and is not indicative of the overall performance of the algorithm in practice. In our analysis we used the fact that a CG step is no worse than a standard gradient step -- a statement that hides the power of the subspace procedure, which is evident in the finite termination result given next.

To prove that algorithms iiCG-1 and iiCG-2 identify the optimal active manifold and the optimal orthant in a finite number of iterations, we assume that strict complementarity holds. Since $v(x^*)=0$, it follows from~\eqref{vdef} that for all $i$ such that $x_i^*=0$ we must have $|g_i (x^*)| \leq \tau$. We say that the solution $x^*$ satisfies strict complementarity if $x_i^*=0$ implies that $|g_i(x^*)| <\tau$. 
\begin{thm}
	\label{thm:termination} If the solution $x^*$ of problem~\eqref{prob} satisfies strict complementarity, then for all sufficiently large $k$, the iterates $x^k$ will lie in the same orthant and active manifold as $x^*$. This implies that iiCG-1 and iiCG-2 identify the optimal solution $x^*$ in a finite number of iterations. 
\end{thm}
\begin{proof}
	We start by defining the sets 
	\begin{align*}
		Z^* = \{i: x^*_i=0\}, \quad N^* = \{i: x^*_i < 0\}, \quad P^* = \{i: x^*_i > 0\}, 
	\end{align*}
	and the constants
	\[ \delta_1 = \min_{i \in N^* \cup P^* }{\frac{|x_i^*|}{2}}, \quad \delta_2 = \min_{i \in Z^*}\left[ \frac{\tau - |g_i(x^*)|}{2}\right]. \]
	Clearly $\delta_1 >0$, and by the strict complementarity assumption we have that $\delta_2 >0$. 
	
	Since, from Theorem~\ref{thm:qlinear} we have that $\{x^k\} \rightarrow x^*$, there exists an integer $k_0$ such that for any $k\geq k_0$ we have 
	\begin{align}
		x^k_i < -\delta_1 \ \ \forall i \in N^*, \quad x^k_i > \delta_1 \ \ \forall i \in P^* \notag \\
		 |x^k_i| < \frac{\alpha \delta_2 }{2} \ \ \forall i \in Z^* \label{secondIneqq} \\
		|g_i| < \tau - \delta_2 \ \ \forall i \in Z^*. \label{thirdIneqq}  
	\end{align}
	
	For all such $k$, we have, by~\eqref{thirdIneqq} that $\omega=0$, which implies the behavior of iiCG-1 and iiCG-2 is identical.  
	
	Thus, all variables that are positive at the solution will be positive for $k > k_0$; and similarly for all negative variables. For the rest of the variables, we consider the ISTA step,
	\[ x^{k+1} = \max \{ |x^k - \alpha  g | - \alpha  \tau, 0 \} \sgn (x^k - \alpha  g ). \]
	Using~\eqref{secondIneqq} and~\eqref{thirdIneqq}, we have that for any $k \geq k_0$ and $i \in Z^*$, 
	\begin{align*}
		| x_i^k - \alpha  g_i | - \alpha  \tau& \leq | x_i^k |+ |\alpha  g_i | - \alpha  \tau\\
		& \leq\frac{\alpha \delta_2 }{2}+ \alpha  ( \tau - \delta_2) - \alpha  \tau\\
		& = - \frac{\alpha \delta_2 }{2} < 0 . 
	\end{align*}
	Therefore, for all $i \in Z^*$ and all $k \geq k_0$, the ISTA step sets $x^{k+1}=0$.
	
	An ISTA step must be taken within $n$ iterations of $k_0$, because of the finite termination property of the conjugate gradient algorithm. Therefore there exists a $k_1$ such that for any $k \geq k_1$, $\sgn(x^k)=\sgn(x^*)$, and by~\eqref{thirdIneqq}, $\omega=0$. These two facts imply that for $ k \geq k_1$, once the algorithm enters the CG iteration it will not leave, since the two {\tt break} conditions cannot be satisfied. Finite termination of CG implies the optimal solution $x^*$ will be found in a finite number of iterations. 
\end{proof}

		\section{Numerical Results} \label{numer} \setcounter{equation}{0}
          
          We developed a MATLAB implementation\footnote{Available at \url{https://github.com/stefanks/Ql1-Algorithm}} of algorithms
		iiCG-1 and iiCG-2, and in the next subsection we compare their performance relative to two well known proximal gradient methods. This allows us to study the algorithmic components of our methods in a controlled setting, and to identify their most promising characteristics. Then we compare, in Section~\ref{establish}, our algorithms to four state-of-the-art codes for solving problem~\eqref{prob}. Our numerical experiments are performed on four groups of test problems with varying characteristics.  

		Before describing the numerical tests, we introduce the following heuristic that improves the prediction made by ISTA steps. As suggested by Wright et al.~\cite{sparsa}, the Barzilai-Borwein stepsize with a non-monotone line-search is usually preferable to a constant stepsize scheme, such as~${\alpha }$ in algorithms~iiCG-1 and~iiCG-2. Thus Line~3 in~iiCG-1 and Lines~4 and~6 in~iiCG-2 are replaced by the following procedure, where we now write $\psi$ in the form $\psi(x^k;\alpha _{ B} )$ to make its dependence on the steplength $\alpha_B$ clear.
 
		\bigskip 
		\textbf{ISTA-BB-LS Step} 
		\begin{algorithmic}
			[1] \STATE $\alpha _{ B} = \fraction{{(x^k-x^{k-1})}^T(x^k-x^{k-1})}{{(x^k-x^{k-1})}^TA(x^k-x^{k-1})}$
			 \REPEAT 
			 			 \STATE$x_F= x^k - \alpha _{ B} \omega(x^k) - \alpha _{ B} \psi(x^k;\alpha _{ B} ) $ (when the ISTA step is invoked) \\ \quad or \\
						 $x_F= x^k - \alpha _{ B} \psi(x^k;\alpha _{ B} ) $ (when the reduced ISTA step is invoked)
			 \STATE $\alpha _B=\fraction{\alpha _B}{2}$ 
			 \UNTIL {$F(x_F) \leq \max_{i \in \{1\ldots M\}} F^i - \alpha _B \xi \|x-x_F \|^2 $} 
			 \STATE $F^{i+1} =F^i$ for all $i \in \{1\ldots M-1 \}$ 
			 \STATE $F^1 =F(x_F)$ 
			 \STATE $x^{k+1}=x_F$ 
		\end{algorithmic}
		\medskip 
		\noindent
		
		At the beginning of the overall algorithm, we initialize $M=5$, $\xi = 0.005$, and $F^i= F(x^0)$ for $i \in \{1\ldots M\}$. The choice of parameters $M, \xi$ is as in~\cite{sparsa}; we did not attempt to fine tune them to our test set. 

		\subsection{Initial Evaluation of the Two New Algorithms} 
		We implemented the following methods. 
		\begin{itemize}
			\item[] \textbf{iiCG-1}
			\item[] \textbf{iiCG-2}
			\item[] \textbf{FISTA} The Fast Iterative Shrinkage-Thresholding Algorithm~\cite{fista}, using a constant stepsize given by $1/L$. 
			\item[] \textbf{ISTA-BB-LS} This method is composed purely of the ISTA-BB-LS steps described at the beginning of this section, which are repeated until convergence. 
		\end{itemize}
	
		These four methods allow us to perform a per-iteration comparison of the progress achieved by each method. {\sc FISTA} and  {\sc ISTA-BB-LS}  are known to be efficient in practice and serve as a useful benchmark. In algorithms iiCG-1 and iiCG-2 we set $c=10^{-4}$, and set 
		$\alpha  = 1/L$ in~\eqref{phitildedef} when computing the value of $\psi (x^k)$ used in the gradient balance condition~\eqref{gamma}. As illustrated in Appendix~\ref{appendix:alphahat}, our algorithms are fairly insensitive to the choice of this parameter. 

		The first three test problems have the following form, which is sometimes called the elastic net regularization problem~\cite{zou_regularization_2005}, 
		\begin{equation} 
			\min_x \fraction{1}{2}\|y-B x \|^2 +\gamma \| x\|^2 + \tau \| x\|_1 . 
			\label{elasticnet}
		\end{equation}
		The data $y$ and $B$ was obtained from three different data sets that we call {\tt spectra}, {\tt sigrec}, and {\tt myrand}. The sources of these data sets are as follows. 

		\medskip 
		\noindent\textbf{Spectra}. The gasoline spectra problem is a regularized linear regression problem~\cite{kalivas_two_1997}; it is available in MATLAB by typing \texttt{load spectra}. This problem has a slightly different form than~\eqref{elasticnet} in the sense that $\ell_1$ regularization is imposed on all but one of the variables (which represents the constant term in linear regression). 

		\medskip 
		\noindent\textbf{Sigrec}. This signal recovery problem is described by Wright et al.~\cite{sparsa}. The authors generate random sparse signals, introduce noise, and encode the signals in a lower dimensional vector by applying a random matrix multiplication. We generated an instance using the code by the authors of~\cite{sparsa}.

		\medskip 
		\noindent\textbf{Myrand}. We generated a random $2000$ variable problem using the following MATLAB commands 
		\begin{verbatim}
			B = randn(1000,2000); y = 2000*randn(1000,1), 
		\end{verbatim}
		and employed this matrix and vector in~\eqref{elasticnet}. \medskip

		\medskip 
		The 4th problem in our test set is of the form 
		\begin{equation}
			 \min_x \fraction{1}{2}x^T Cx -y^Tx +\gamma{ \| x \|}^2 + \tau{ \| x \|}_1 . 
			\label{prox-prob}
		\end{equation}
		\noindent\textbf{Proxnewt}. The data for~\eqref{prox-prob}  was generated by applying the proximal Newton method described in~\cite{istanbul} to an $\ell_1$ regularized convex optimization problem of the form $\varphi(x) + \tau \|x\|_1$, where $\varphi(x)$ is a logistic regression function and the data is given by the {\tt gisette} test set in the LIBSVM repository~\cite{CC01a}. Each iteration of the proximal Newton method computes a step by solving a subproblem of the form~\eqref{prob}. We extracted one of these subproblems, and added the $\ell_2$ regularization term $\gamma \|x\|^2$ to yield a problem of the form~\eqref{prox-prob}.

		\bigskip 

		We created twelve versions of each of the four problems listed above, by choosing different values of $\gamma$ and $\tau$. This allowed us to create problems with various degrees of ill conditioning and different percentages on non-zeros in the solution. In our datasets, the matrices $B$ and $C$ in~\eqref{elasticnet} and~\eqref{prox-prob} are always rank deficient; therefore, when $\gamma=0$ the Hessians of $f(x)$ are singular.

		The following naming conventions are used. The last digit, as in problems

		\smallskip\centerline{ spectra{s}{\bf 1}, $\cdots$ , spectra{s}{\bf 4}, } \smallskip 
		\noindent indicates one of the four values of $\tau$ that were chosen for each problem so as to generate different percentages of non-zeros in the optimal solution. The degree of ill conditioning, which is controlled by $\gamma$, is indicated in the second-to-last character, as in 

		\centerline{ spectra{\bf s}{1}, \, spectra{\bf i}{1}, \, spectra{\bf m}{1}, } \smallskip 
		\noindent which correspond to the \textbf{s}ingular, \textbf{i}ll conditioned and \textbf{m}oderately conditioned versions of the problem. The characteristics of the test problems are given in Appendix~\ref{appendix:sparsitypatterns}. 

		Accuracy in the solution is measured by the ratio 
		\begin{equation}
			{\tt tol} =  \frac{F(x^k) - F^*}{|F^*|} , 
			\label{accu} 
		\end{equation}
		where $F^*$ is the best known objective value for each problem. Given the nature of the four algorithms listed above, it is easy to compute and report the ratio~\eqref{accu} after each matrix-vector product computation. We initialized $x^0$ to the zero vector, and imposed a limit of 50,000 matrix-vector products on all the runs.

		In Table~\ref{bigtable} we report the results for iiCG-1, iiCG-2, FISTA and ISTA-BB-LS on all the test problems. 
Matrix-Vector product (MV) counts are a reasonable measure of computational work used by these algorithms since they are by far the costliest operations. All algorithms are terminated as  soon as the ratio in~\eqref{accu} is less than  a prescribed constant;
in Table~\ref{bigtable} we report results for {\tt tol}$ =10^{-4}$, and {\tt tol} $=10^{-10}$. 
		Dashes signify failures to find a solution after 50,000 matrix-vector products, and bold numbers mark the best-performing algorithm. 
		
		We observe from these tables that problems with an intermediate value of $\tau$ (typically) require the largest effort. This suggests that the values of $\tau$ chosen in our tests gave rise to an interesting collection of problems that range from nearly quadratic to highly regularized piecewise quadratic, with the most challenging problems in the middle range. An analysis of the data given in Table~\ref{bigtable} indicates that iiCG-2 is the most efficient in these tests, but not uniformly so. Overall, we regard both iiCG-1 and iiCG-2 as promising methods for solving the regularized quadratic problem~\eqref{prob}. 

		\begin{table}
			\caption{Number of matrix-vector products to reach accuracies {\tt tol} $ =10^{-4}$ and $ =10^{-10}$.}
			{
				\begin{tabular}{|c|cccc|cccc|}
					\hline
					 & \multicolumn{4}{|c|}{{\tt tol} $ =10^{-4}$} &  \multicolumn{4}{|c|}{{\tt tol} $ =10^{-10}$}\\
					               & iiCG-1 & iiCG-2 & FISTA& ISTA-BB-LS& iiCG-1 & iiCG-2 & FISTA& ISTA-BB-LS\\
					 \hline
		{\tt myrands1} & 654 & 297 & \textbf{248} & 1311  & 13614 & \textbf{8102} & 8511 & - \\ 
		{\tt myrands2} & 513 & 310 & \textbf{163} & 547  & 3228 & \textbf{1885} & 4748 & 12782 \\ 
		{\tt myrands3} & 118 & 123 & \textbf{69} & 86  & 398 & \textbf{311} & 1493 & 540 \\ 
		{\tt myrands4} & 11 & 12 & 21 & \textbf{8}  & 18 & 20 & 106 & \textbf{17} \\ 
		{\tt myrandi1} & \textbf{14} & \textbf{14} & 31 & 19  & - & - & \textbf{26183} & - \\ 
		{\tt myrandi2} & 491 & 297 & \textbf{163} & 498  & 3008 & \textbf{1912} & 4925 & 13682 \\ 
		{\tt myrandi3} & 111 & 116 & \textbf{69} & 86  & 352 & \textbf{335} & 1492 & 596 \\ 
		{\tt myrandi4} & 11 & 12 & 21 & \textbf{8}  & 18 & 20 & 106 & \textbf{17} \\ 
		{\tt myrandm1} & \textbf{14} & \textbf{14} & 30 & 19  & \textbf{57} & \textbf{57} & 236 & 726 \\ 
		{\tt myrandm2} & 121 & \textbf{108} & 111 & 317  & 1731 & \textbf{728} & 2437 & 5483 \\ 
		{\tt myrandm3} & 117 & 128 & \textbf{68} & 86  & 375 & \textbf{359} & 1433 & 466 \\ 
		{\tt myrandm4} & 11 & 12 & 21 & \textbf{8}  & 18 & 19 & 106 & \textbf{17} \\
		{\tt spectras1 }& \textbf{4} & \textbf{4} & 265 & 17  & - & \textbf{45888} & - & - \\  
		{\tt spectras2 }& \textbf{4} & \textbf{4} & 264 & 20  & 48200 & \textbf{8656} & - & - \\  
		{\tt spectras3 }& \textbf{4} & \textbf{4} & 263 & 26  & 5661 & \textbf{2245} & - & - \\  
		{\tt spectras4 }& \textbf{4} & \textbf{4} & 270 & 22  & 30896 & \textbf{9170} & - & - \\  
		{\tt spectrai1 }& \textbf{4} & \textbf{4} & 258 & 23  & \textbf{42} & \textbf{42} & 29258 & 12046 \\  
		{\tt spectrai2 }& \textbf{4} & \textbf{4} & 257 & 26  & 159 & \textbf{129} & 33734 & - \\  
		{\tt spectrai3 }& \textbf{4 }& \textbf{4} & 256 & 19  & 2246 & \textbf{2205} & 45339 & - \\  
		{\tt spectrai4 }& \textbf{60} & 105 & 1036 & 4192  & 1898 & \textbf{1751} & 10030 & 23579 \\  
		{\tt spectram1 }& \textbf{2} & \textbf{2} & \textbf{2} & \textbf{2}  & \textbf{10} & \textbf{10} & 1897 & 17 \\  
		{\tt spectram2 }& \textbf{2} & \textbf{2} &\textbf{ 2} & \textbf{2}  & 15 & \textbf{12} & 2024 & 137 \\  
		{\tt spectram3 }& \textbf{5 }& \textbf{5 }& 51 & 7  & \textbf{11} & \textbf{11} & 1445 & 163 \\  
		{\tt spectram4 }& \textbf{100} & \textbf{100} & 126 & 175  & \textbf{107} & \textbf{107} & 4799 & 545 \\  
		{\tt sigrecs1 }& 2206 & 1213 & \textbf{446} & 3522 & 3635 & 2283 & \textbf{1338} & 6687 \\  
		{\tt sigrecs2 }& 1020 & 589 & \textbf{291} & 1494 & 1191 & \textbf{695} & 696 & 1737 \\  
		{\tt sigrecs3 }& 105 & 94 & 75 & \textbf{72} & 120 & 116 & 296 & \textbf{86} \\  
		{\tt sigrecs4 }& 11 & 12 & 25 & \textbf{8} & 19 & 20 & 145 & \textbf{16} \\  
		{\tt sigreci1 }& 8 & 8 & 14 & 10 & - & 5415 & 10542 & - \\  
		{\tt sigreci2 }& 2148 & 1156 & \textbf{442} & 3515 & 3526 & 2190 & \textbf{1350} & 6955 \\  
		{\tt sigreci3 }& 1095 & 532 & \textbf{291} & 1499 & 1285 & \textbf{637} & 659 & 1728 \\  
		{\tt sigreci4 }& 11 & 12 & 25 & \textbf{8} & 19 & 20 & 145 & \textbf{16} \\  
		{\tt sigrecm1 }& \textbf{8} & \textbf{8} & 14 & 10 & \textbf{51} & \textbf{51} & 114 & 386 \\  
		{\tt sigrecm2 }& 65 & \textbf{60} & 61 & 101 & 777 & \textbf{297} & 864 & 1138 \\  
		{\tt sigrecm3 }& 199 & 137 & \textbf{103} & 229 & 476 & \textbf{365} & 1301 & 504 \\  
		{\tt sigrecm4 }& 11 & 12 & 25 & \textbf{8} & 18 & 19 & 144 & \textbf{16} \\  
		{\tt proxnewts1}& 22739 & \textbf{4077} & 21173 & - & 40139 & \textbf{10169} & - & - \\  
		{\tt proxnewts2}& 9522 & 6871 & \textbf{6423} & 38233 & 15782 & \textbf{8436} & - & - \\  
		{\tt proxnewts3}& 6463 & \textbf{1865} & 2026 & 3659 & 7092 & \textbf{2098} & - & 8384 \\  
		{\tt proxnewts4}& 212 & \textbf{191} & 1582 & 311 & 233 & \textbf{213} & - & 458 \\  
		{\tt proxnewti1}& 294 & \textbf{336} & 2795 & 49932 & 3267 & \textbf{1100} & - & - \\  
		{\tt proxnewti2}& 2274 & \textbf{1384} & 2212 & 14029 & 3519 & \textbf{1983} & - & 24756 \\  
		{\tt proxnewti3}& 6076 & \textbf{1437} & 1702 & 3027 & 6585 & \textbf{1647} & - & 6344 \\  
		{\tt proxnewti4}& 291 & \textbf{160} & 1499 & 296 & 315 & \textbf{175} & - & 463 \\  
		{\tt proxnewtm1}& \textbf{32} & \textbf{32} & 881 & 190 & 131 & \textbf{112} & 20319 & 2382 \\  
		{\tt proxnewtm2}& 41 & \textbf{36} & 784 & 293 & 139 & \textbf{101} & 14310 & 1369 \\  
		{\tt proxnewtm3}& 237 & \textbf{232} & 592 & 492 & \textbf{262} & 274 & 12640 & 720 \\  
		{\tt proxnewtm4}& 58 & \textbf{50} & 472 & 101 & 70 & \textbf{58} & 10380 & 128 \\  
					\hline
			    \end{tabular}
			}
			\label{bigtable}
		\end{table}

		Using the data from Table~\ref{bigtable}, we illustrate in Figure~\ref{fig:Res2} the relative performance of iiCG-1, iiCG-2, FISTA and ISTA-BB-LS, using the Dolan-Mor\'e profiles~\cite{DolMor01} (based on the number of matrix-vector multiplications required for convergence). While iiCG-1 and iiCG-2 are efficient in the case {\tt tol} $= 10^{-4}$, iiCG-2 demonstrates superior performance in reaching the higher accuracy.
		\begin{figure}
			 \scriptsize \centering 
			\begin{subfigure}
				[b]{0.49 
				\textwidth} \includegraphics[scale=0.4]{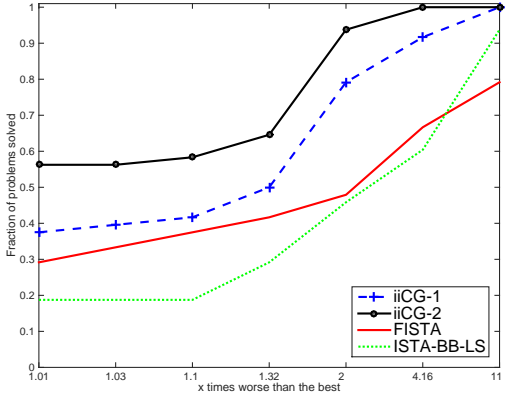} \caption{tol = 1e-4} 
			\end{subfigure}
			\begin{subfigure}
				[b]{0.49 
				\textwidth} \includegraphics[scale=0.4]{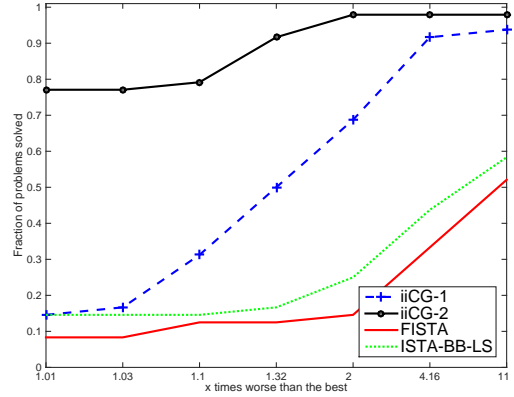}\caption{tol = 1e-10} 
			\end{subfigure}
			\caption{Comparison of the four algorithms using the logarithmic Dolan-Mor\'e profiles, based on the number of matrix-vector products required for convergence. We report results for low and high accuracy~\eqref{accu} in the objective function.} \label{fig:Res2} 
		\end{figure}

		In Figure~\ref{fig:Res3} we illustrate typical behavior of iiCG-1 and iiCG-2 by means of problems {\tt proxnewts3} and {\tt spectram4}. We plot the accuracy in the objective~\eqref{accu} as a function of of matrix-vector multiplications.  Both plots show that our algorithms are able to estimate the solution to high accuracy. They sometimes outperform the other methods from the very start, as in Figure~\ref{prsp1}, but in other cases iiCG-1 and iiCG-2 show their strength later on in the runs; see Figure~\ref{prsp2}. We note that the ISTA-BB-LS method is more efficient than FISTA when high accuracy in the solution is required. 

		\begin{figure}
			 \scriptsize \centering 
			\begin{subfigure}
				[b]{0.49 
				\textwidth} \includegraphics[scale=0.4]{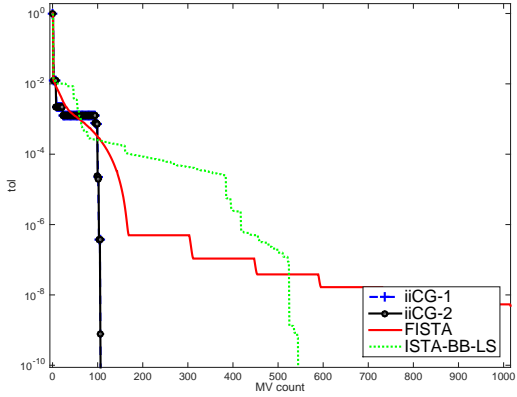} \caption{Problem {\tt spectram4}} \label{prsp1} 
			\end{subfigure}
			\begin{subfigure}
				[b]{0.49 
				\textwidth} \includegraphics[scale=0.4]{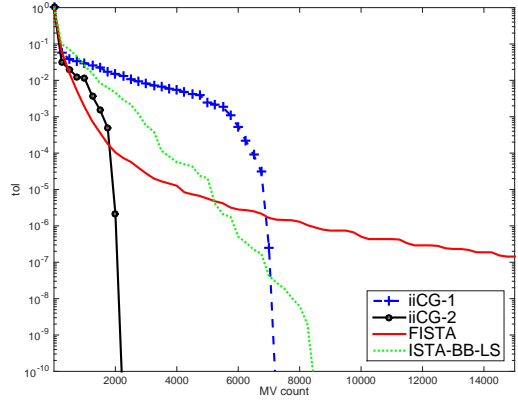} \caption{Problem {\tt proxnewts3}} \label{prsp2} 
			\end{subfigure}
			\caption{Accuracy~\eqref{accu} in the objective function (vertical axis) as a function of the number of matrix-vector products performed (MV count) performed by each algorithm} \label{fig:Res3} 
		\end{figure}
		
		\newpage

		We have observed that Algorithm iiCG-2 is superior to iiCG-1 in identifying sparse solutions, due to its judicious application of the subspace ISTA iteration. To illustrate this, we plot in Figure~\ref{fig:superSparsity} the Pareto frontier based on two quantities: the accuracy \eqref{accu} in the solution, and the number of nonzero elements in the solution. Specifically, we ran the test problems, {\tt spectras2} and {\tt myrands2} until a specified  limit of matrix-vector (MV) products was computed (500, 2100, 30000 for {\tt spectras2}  and 500, 1000, 2000 for {\tt myrands2}). For each run, all iterates were considered, and we plotted the pairs (accuracy-nonzeros) such that no other pair existed with both higher accuracy and sparsity. As expected, when more effort (MV) is allowed, higher accuracy is achieved, but not necessarily higher sparsity. As measured by the two quantities depicted in Figure~\ref{fig:superSparsity}, iiCG-2 finds better solutions.
		
		\begin{figure}
			 \scriptsize \centering 
			\begin{subfigure}
				[b]{0.49 
				\textwidth} \includegraphics[scale=0.4]{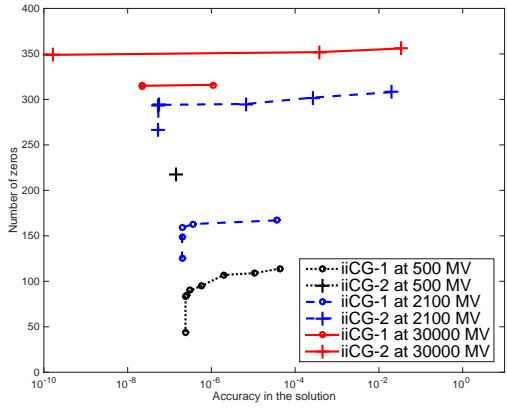} \caption{On problem {\tt spectras2} }
			\end{subfigure}
			\begin{subfigure}
				[b]{0.49 
				\textwidth} \includegraphics[scale=0.4]{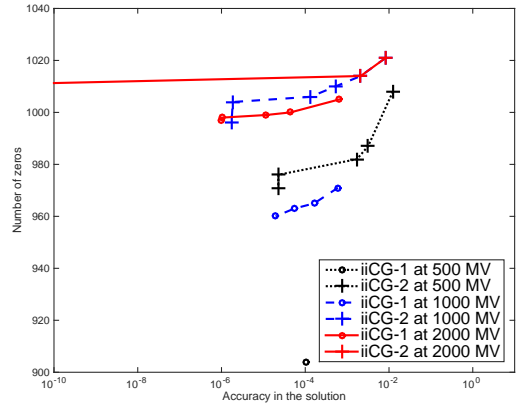}\caption{On problem {\tt myrands2}} 
			\end{subfigure}
			\caption{Pareto frontier based on the number of nonzeros in the incumbent solution (vertical axis) and accuracy in the solution \eqref{accu} (horizontal axis), for runs imposing 3 limits on the maximum number of matrix-vector (MV) products. The test problems,
			{\tt spectras} and {\tt myrands2}, represent typical behavior of the algorithms.} \label{fig:superSparsity} 
		\end{figure}

		\newpage
		
		\subsection{Analysis of the CG Phase}

		We now discuss the behavior of the subspace CG phase, which has a great impact on the overall efficiency of the proposed algorithms. In Figure~\ref{fig:Res4} we report data for two representative runs of iiCG-2 given by test problems {\tt myrandm1} and {\tt sigreci4}. The horizontal axis labels each of the subspace phases invoked by the algorithm, and the vertical axis gives the number of CG iterations performed during that subspace phase.

		\bigskip 
		\begin{figure}[htp]
			 \scriptsize \centering 
			\begin{subfigure}
				[b]{0.49 
				\textwidth} \includegraphics[scale=0.7]{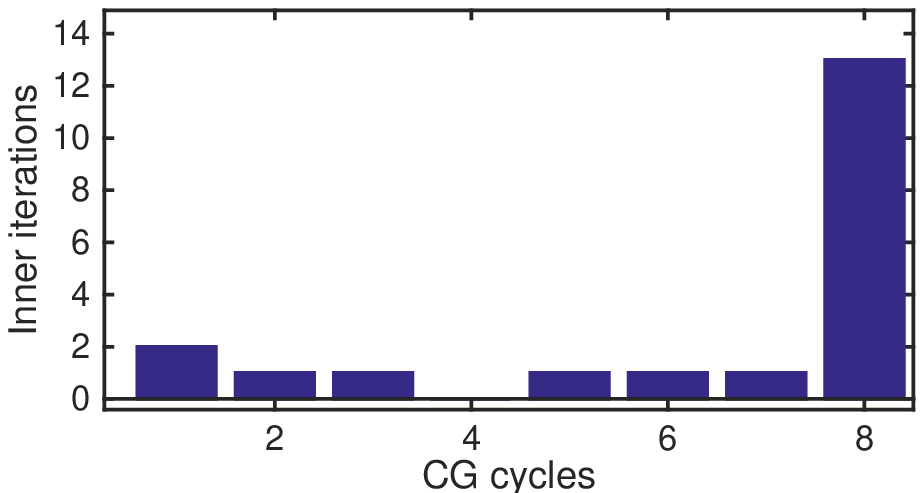} \caption{iiCG-2 on {\tt sigreci4} } \label{fig:Res42} 
			\end{subfigure}
			\begin{subfigure}
				[b]{0.49 
				\textwidth} \includegraphics[scale=0.7]{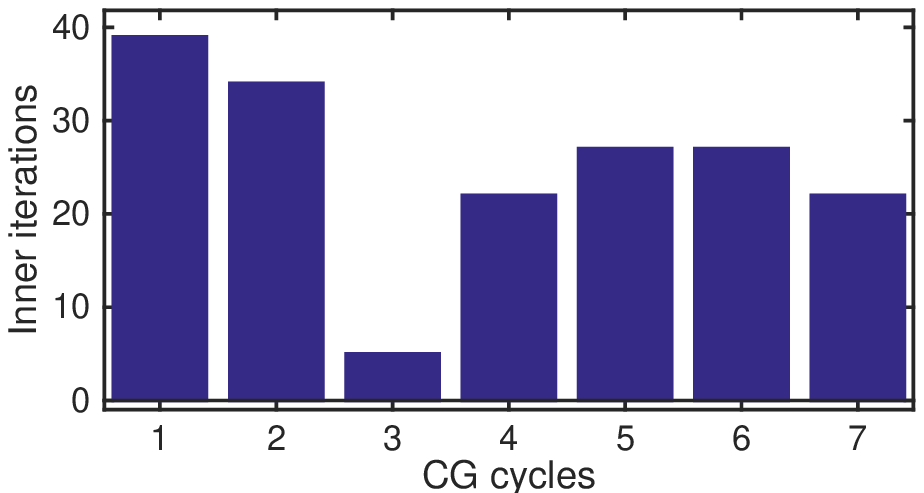}\caption{iiCG-2 on {\tt myrandm1}} \label{fig:Res41} 
			\end{subfigure}
			\caption{Number of CG iterations in each subspace phase} \label{fig:Res4} 
		\end{figure}
		\bigskip 

		Figure~\ref{fig:Res42} illustrates a behavior that is often observed for moderate and large values of the penalty parameter $\tau$, namely that the bulk of the CG iterations are performed towards the end of the run. This is desirable, as the CG phase makes a moderate contribution earlier on towards identifying the optimal active set, and is then able to compute a highly accurate solution of the problem in one (or two) CG cycles. Figure~\ref{fig:Res41}, considers the case when the penalty parameter $\tau$ is very small, i.e., when $F$ is nearly quadratic. We observe now that the effort expended by the CG phase is more evenly distributed throughout the run. It is reassuring that the number of CG iterations does not tail off for this problem, and that a significant number of CG steps is performed in the last cycle, yielding an accurate solution to the problem. 

		\subsection{Comparisons with Established Codes}   \label{establish}

		We also performed comparisons with the following four state-of-the-art codes. To facilitate our comparisons, and ease of implementation, we only considered codes written in MATLAB. 
		\begin{itemize}
			\item \textbf{SPARSA} This is the well known implementation of the {\sc ISTA} method described in~\cite{sparsa}. The code can be found at \url{http://www.lx.it.pt/~mtf/SpaRSA/} 
			\item \textbf{PSSgb} The  algorithm implemented in this code is motivated by the two-metric projection method~\cite{gafni1984} for bound constrained optimization. That method is extended to the regularized $\ell_1$ problem; curvature information is incorporated in the form of a BFGS matrix. \url{http://www.di.ens.fr/~mschmidt/Software/thesis.html} 
			\item \textbf{N83} Is one of the codes provided by the TFOCS package~\cite{becker2011templates}. It implements the optimal first order Nesterov method described in~\cite{nesterov83}. \url{http://cvxr.com/tfocs/download/} 
			\item \textbf{pdNCG} A Newton-CG method in which
			the $\ell_1$ norm is approximated by a smooth function~\cite{FountoulakisGondzio2014}. \url{http://www.maths.ed.ac.uk/~kfount/index.html}
		\end{itemize}
		We also experimented with SALSA~\cite{salsa}, TWIST~\cite{twist} and FPC\_AS~\cite{hale_fixed-point_2007}, l1\_ls~\cite{l1ls}, YALL1~\cite{YALL1}, but these codes were not competitive on our test set with the four packages mentioned above.

		In Figure~\ref{fig:dm} we compare our algorithms with the four codes listed above on all test problems. The figure plots the Dolan-Mor\'e performance profiles based on CPU time; we report results for two values of the convergence tolerance~\eqref{accu}. Figure~\ref{fig:dm} indicates that PSSgb is the closest in performance compared to our algorithms. 
		\begin{figure}[htp]
			 \centering \scriptsize 
			\begin{subfigure}
				[b]{0.49 
				\textwidth} \includegraphics[scale=0.4]{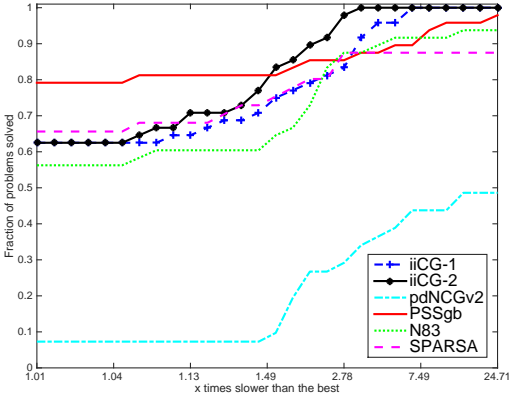} \caption{${\tt tol}=10^{-4}$}   
			\end{subfigure}
			\begin{subfigure}
				[b]{0.49 
				\textwidth} \includegraphics[scale=0.4]{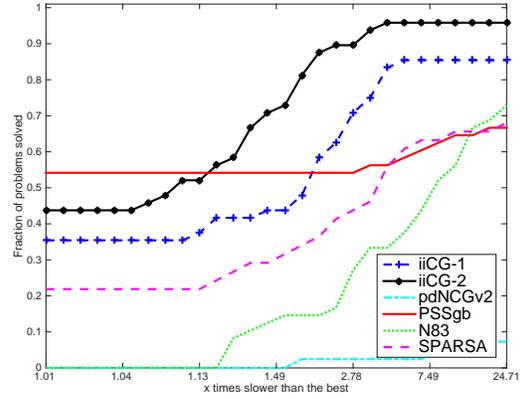}\caption{${\tt tol}=10^{-10}$} 
			\end{subfigure}
			\caption{Comparison of iiCG-1, iiCG-2, SPARSA, N83 and PSSgb. The figure plots the logarithmic Dolan-Mor\'e performance profiles based on CPU time for problems {\tt spectra}, {\tt sigrec} and {\tt myrand}. } \label{fig:dm} 
		\end{figure}

		We now examine the behavior of the codes for intermediate values of accuracy. Figure~\ref{efficc} shows the fraction of problems that a method was able to solve faster than the other methods, as a function of the accuracy measure~\eqref{accu}, in a log-scale.

		\begin{figure}
		\centering
			\includegraphics[scale=0.4]{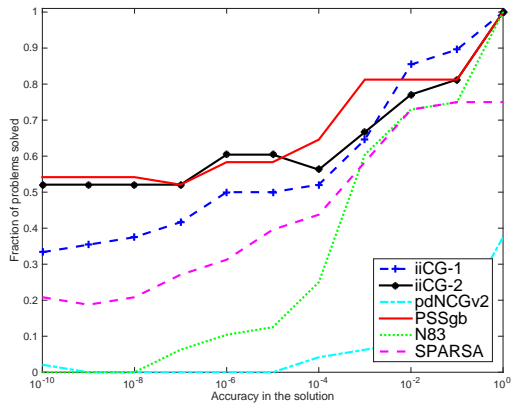} 
				\caption{Efficiency. For given accuracy in the function value (horizontal axis), the plot shows the percentage of problems solved to that accuracy within 10\% of the time of the best method.} 
				 \label{efficc} 
		\end{figure}

\newpage
		Some algorithms are better suited for problems with a very high solution sparsity. In Figure~\ref{fig:spar}, we show the Dolan-Mor\'e profiles for two sets of problems; one with low sparsity in the solution, and one with high sparsity. The first set of test problems is composed of problems styled \emph{name1} (average solution sparsity 25\%), and the second is made of problems styled \emph{name4} (average solution sparsity 95\%). We observe that iiCG-2 outperforms the other codes for low solution sparsity, and that PSSgb is competitive for high sparsity problems. We did not include pdNCGv2 and SPARSA since they are not competitive with the other algorithms, in these tests.

Overall, our iiCG methods are competitive with the four state-of-the-art codes in these tests.

		\begin{figure}[htp]
			 \scriptsize \centering 
			\begin{subfigure}
				[b]{0.49 
				\textwidth} \includegraphics[scale=0.4]{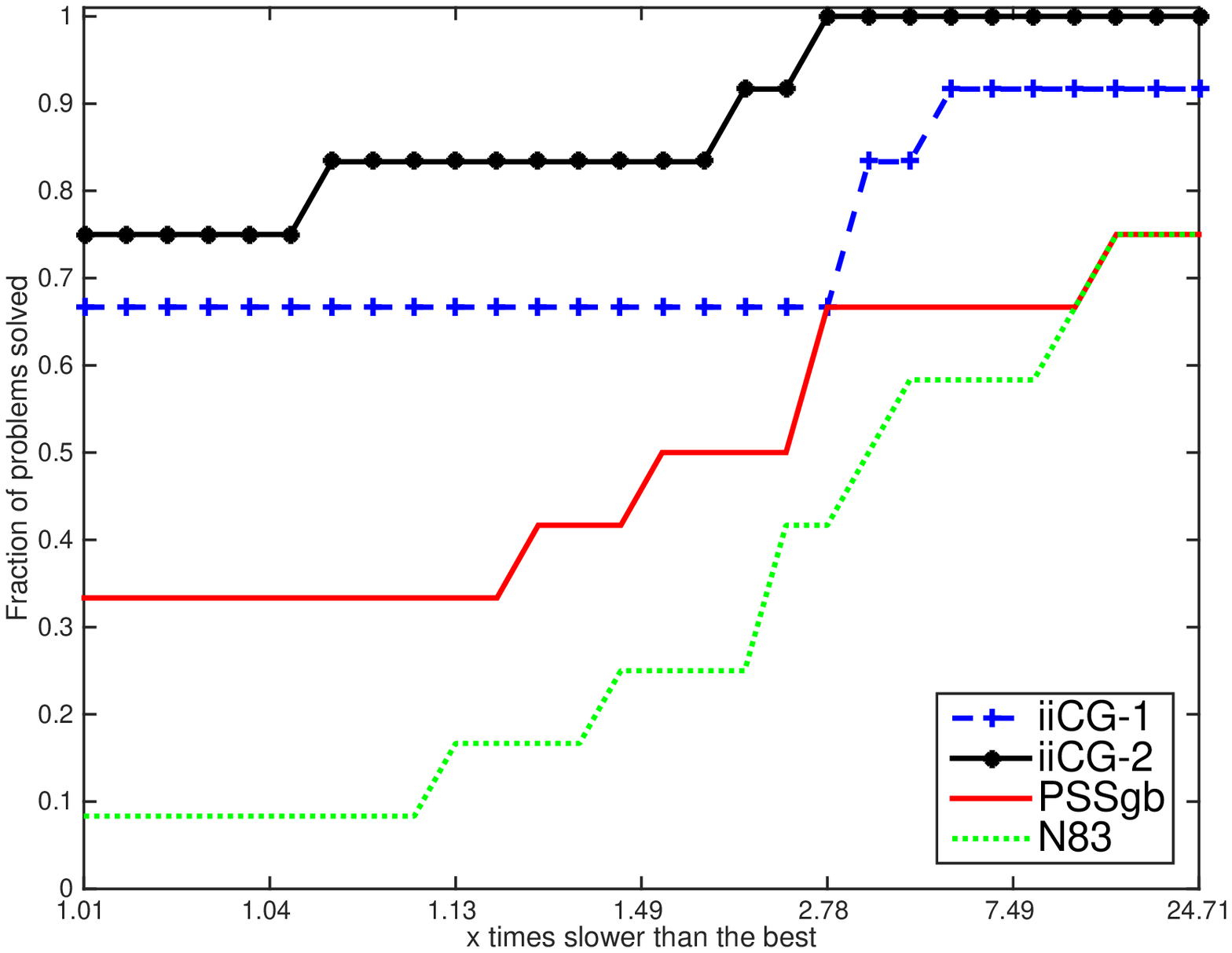} \caption{Low Sparsity} 
			\end{subfigure}
			\begin{subfigure}
				[b]{0.49 
				\textwidth} \includegraphics[scale=0.4]{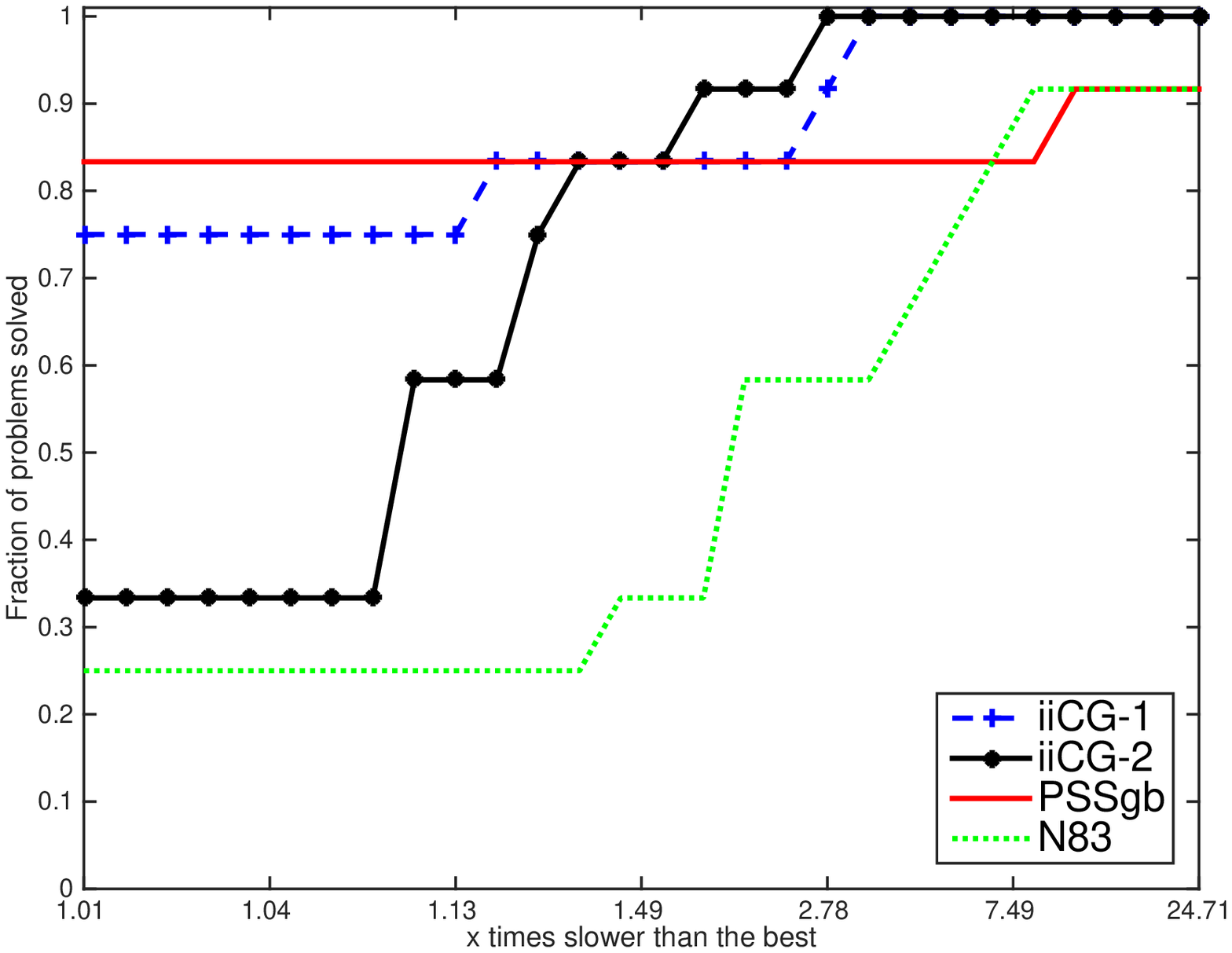}\caption{High Sparsity} 
			\end{subfigure}
			\caption{Comparison of four algorithms using the logarithmic Dolan-Mor\'e profiles, for problems with: low sparsity (Figure a), and high sparsity (Figure b). We set $tol=10^{-7}$} \label{fig:spar} 
		\end{figure}

\section{Final Remarks} \label{remarks}

In this paper, we presented a second-order method for solving the $\ell_1$ regularized quadratic problem~\eqref{prob}. We call the method iiCG, for ``interleaved ISTA-CG'' method. It differs from other second-order methods proposed in the literature in that it can invoke one of two possible steps at each iteration: a subspace conjugate gradient step or a first order active-set identification step. This flexibility is designed to allow the algorithm to adapt itself to the problem to be solved, and the results presented in the paper suggest that it is generally successful at achieving this goal. The decision of what type of step to invoke is based on a measure related to the relative components of the minimum norm subgradient --- an idea proposed by Dostal and Schoeberl~\cite{dostal_minimizing_2005} for the solution of bound constrained quadratic optimization problems. But unlike~\cite{dostal_minimizing_2005}, our algorithm does not include a  step along the direction $-\omega(x^k)$ in order to relax zero variables; as a result our approach departs significantly from the methodology given in~\cite{dostal_minimizing_2005}.

We presented two variants of our method that differ in the first-order active set identification phase. One option uses ISTA, and the other a subspace ISTA iteration.

To provide a theoretical foundation for our method, we established global rates of convergence and complexity bounds based on the total amount of work expended by the algorithm (measured by the total number of matrix-vector products performed). We also gave careful consideration to the main design components of the algorithm,  such as the definition of the vectors $\omega, \psi$ employed in the gradient balance condition~\eqref{gamma}. One of the features of the algorithm that was particularly successful is allowing the CG iteration to cross orthants as long as sufficient decrease in the objective function is achieved. The numerical tests reported in this paper suggest that the algorithm proposed in this paper is competitive with state-of-the-art codes. 

\bibliographystyle{plain} 
\bibliography{paper.bbl}

\newpage

\section{Dataset Details and Sparsity Patterns} \label{appendix:sparsitypatterns}

The $\tau$ values for each problem were chosen by experimentation so as to span a range of solution sparsities. This is preferable to setting $\tau$ as a multiple of $\|b\|_\infty$ (as is often done in the literature based on the fact when $\tau = \|b\|_\infty$ the optimal solution to problem~\eqref{prob} is the zero vector~\cite{fuchs_more_2004}). We prefer to select the value of $\tau$ for each problem, as there sometimes is a very small range of values  that yields interesting problems.
\begin{table}
	 \caption{{\tt myrand} $n=2000$}
	{\begin{tabular}{|c|c|c|c|c|c|} 
		\hline 
		$norm(A)$ & $cond(A)$ & $\gamma$ & problem & $\tau$& num zeros\\
		\hline\multirow{4}{*}{5781.5}&\multirow{4}{*}{-}&\multirow{4}{*}{0}&{\tt myrands1}&100&1001\\
		& & &{\tt myrands2}&1000&1011\\
		& & &{\tt myrands3}&10000&1133\\
		& & &{\tt myrands4}&100000&1850\\
		\hline \multirow{4}{*}{5781.5}&\multirow{4}{*}{5.7815e+06}&\multirow{4}{*}{0.001}&{\tt myrandi1}&0.1&583\\
		& & &{\tt myrandi2}&100&1011\\
		& & &{\tt myrandi3}&10000&1133\\
		& & &{\tt myrandi4}&100000&1850\\
		\hline \multirow{4}{*}{5782.5}&\multirow{4}{*}{5782.5}&\multirow{4}{*}{1}&{\tt myrandm1}&0.1&4\\
		& & &{\tt myrandm2}&100&620\\
		& & &{\tt myrandm3}&10000&1130\\
		& & &{\tt myrandm4}&100000&1850\\
		\hline 
	\end{tabular}}
	 \label{table:myrand} 
\end{table}
\begin{table}
	\caption{{\tt spectra }$n=402$} 
	{\begin{tabular}{|c|c|c|c|c|c|}
		 \hline 
		 $norm(A)$ & $cond(A)$ & $\gamma$ & problem & $\tau$& num zeros\\
		\hline\multirow{4}{*}{2.056413e+03}&\multirow{4}{*}{-}&\multirow{4}{*}{0}&{\tt spectras1}&1.0e-06&322\\
		& & &{\tt spectras2}&1.0e-04&348\\
		& & &{\tt spectras3}&1.0e-03&372\\
		& & &{\tt spectras4}&1.0e-02&389\\
		\hline \multirow{4}{*}{2.056414e+03}&\multirow{4}{*}{2.056414e+06}&\multirow{4}{*}{1.0e-03}&{\tt spectrai1}&3.0e-05&2\\
		& & &{\tt spectrai2}&1.0e-03&91\\
		& & &{\tt spectrai3}&1.0e-02&313\\
		& & &{\tt spectrai4}&5.0e-01&398\\
		\hline \multirow{4}{*}{2.057413e+03}&\multirow{4}{*}{2.057413e+03}&\multirow{4}{*}{1}&{\tt spectram1}&1.0e-03&1\\
		& & &{\tt spectram2}&2.0e-01&109\\
		& & &{\tt spectram3}&1&332\\
		& & &{\tt spectram4}&30&388\\
		\hline 
	\end{tabular}}
	\label{table:spectra} 
\end{table}
\begin{table}
	 \caption{{\tt sigrec }$n=4096$}
	{\begin{tabular}{|c|c|c|c|c|c|}
		 \hline $norm(A)$ & $cond(A)$ & $\gamma$ & problem & $\tau$& num zeros\\
		\hline\multirow{4}{*}{1.119904e+00}&\multirow{4}{*}{-}&\multirow{4}{*}{0}&{\tt sigrecs1}&5.0e-05&3549\\
		& & &{\tt sigrecs2}&2.0e-04&3816\\
		& & &{\tt sigrecs3}&5.0e-03&3860\\
		& & &{\tt sigrecs4}&1.0e-01&3973\\
		\hline \multirow{4}{*}{1.119905e+00}&\multirow{4}{*}{1.119905e+06}&\multirow{4}{*}{1.0e-06}&{\tt sigreci1}&5.0e-08&828\\
		& & &{\tt sigreci2}&5.0e-05&3535\\
		& & &{\tt sigreci3}&2.0e-04&3813\\
		& & &{\tt sigreci4}&1.0e-01&3973\\
		\hline \multirow{4}{*}{1.120904e+00}&\multirow{4}{*}{1.120904e+03}&\multirow{4}{*}{1.0e-03}&{\tt sigrecm1}&4.5e-07&16\\
		& & &{\tt sigrecm2}&1.0e-04&1519\\
		& & &{\tt sigrecm3}&2.0e-03&3310\\
		& & &{\tt sigrecm4}&1.0e-01&3973\\
		\hline 
	\end{tabular}}
	 \label{table:sigrec} 
\end{table}
\begin{table}
	 \caption{{\tt proxnewt} $n=5000$} 
	{\begin{tabular}{|c|c|c|c|c|c|}
		 \hline $norm(A)$ & $cond(A)$ & $\gamma$ & problem & $\tau$& num zeros\\
		\hline\multirow{4}{*}{1.103666e+02}&\multirow{4}{*}{-}&\multirow{4}{*}{0}&{\tt proxnewts1}&6.7e-06&1893\\
		& & &{\tt proxnewts2}&6.7e-05&3192\\
		& & &{\tt proxnewts3}&6.7e-04&4365\\
		& & &{\tt proxnewts4}&6.7e-03&4960\\
		\hline \multirow{4}{*}{1.103667e+02}&\multirow{4}{*}{1.103667e+06}&\multirow{4}{*}{1.0e-04}&{\tt proxnewti1}&6.7e-06&1395\\
		& & &{\tt proxnewti2}&6.7e-05&3060\\
		& & &{\tt proxnewti3}&6.7e-04&4344\\
		& & &{\tt proxnewti4}&6.7e-03&4959\\
		\hline \multirow{4}{*}{1.103771e+02}&\multirow{4}{*}{1.051211e+04}&\multirow{4}{*}{1.0e-02}&{\tt proxnewtm1}&6.7e-06&193\\
		& & &{\tt proxnewtm2}&6.7e-05&1283\\
		& & &{\tt proxnewtm3}&6.7e-04&3602\\
		& & &{\tt proxnewtm4}&6.7e-03&4926\\
		\hline 
	\end{tabular}}
	\label{table:proxnewt} 
\end{table}

\section{Effect of overestimating $\|A\|$ in the Gradient Balance Condition} \label{appendix:alphahat}
In our experiments, we set $\alpha = 1/L$ in iiCG, in the gradient balance condition computation. Often $L$ is not known and is hard to compute (for medium and large-scale problems computing $L$ may take longer than running the algorithm itself). Figure~\ref{fig:alphahat} shows that while $\alpha = \frac{1}{L}$ is a good choice, iiCG-2 is fairly insensitive to the choice of $\alpha$, particularly if the value $1/L$ is overestimated.

\begin{figure} \scriptsize \centering \includegraphics[scale=0.4]{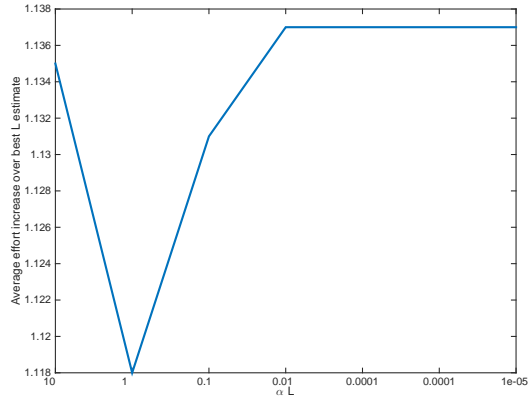} \caption{Average increase in matrix-vector products relative to the optimal choice for $\alpha$ (obtained by experimentation), for various choices of $\alpha$. The results are compiled from all 48 test problems, and the runs were stopped when  {\tt tol}=$10^{-4}$  .} \label{fig:alphahat} 
\end{figure}

\end{document}